\documentclass{amsart}
\usepackage{mathptmx, amssymb, enumerate, colonequals, amscd, bm}
\usepackage{color}
\usepackage[colorlinks=true, pagebackref, hyperindex, citecolor=green, linkcolor=red]{hyperref}
\usepackage{subdepth}

\newtheorem{theorem}{Theorem}[section]
\newtheorem{lemma}[theorem]{Lemma}
\newtheorem{proposition}[theorem]{Proposition}
\newtheorem{corollary}[theorem]{Corollary}

\theoremstyle{definition}
\newtheorem{example}[theorem]{Example}
\newtheorem{definition}[theorem]{Definition}
\newtheorem{remark}[theorem]{Remark}
\newtheorem{question}[theorem]{Question}
\numberwithin{equation}{theorem}

\def\ge{\geqslant}
\def\le{\leqslant}
\def\phi{\varphi}
\def\rho{\varrho}
\def\tilde{\widetilde}
\def\bar{\overline}
\def\to{\longrightarrow}
\def\mapsto{\longmapsto}
\def\into{\lhook\joinrel\longrightarrow}
\def\onto{\relbar\joinrel\twoheadrightarrow}
\def\ann{\operatorname{ann}}
\def\ker{\operatorname{ker}}
\def\coker{\operatorname{coker}}
\def\image{\operatorname{image}}
\def\ffield{\operatorname{frac}}
\def\tr{\operatorname{tr}}
\def\pf{\operatorname{pf}}
\def\Pf{\operatorname{Pf}}
\def\rad{\operatorname{rad}}
\def\End{\operatorname{End}}
\def\Hom{\operatorname{Hom}}
\def\Spec{\operatorname{Spec}}

\def\GL{\operatorname{GL}}
\def\Sp{\operatorname{Sp}}
\def\O{\operatorname{O}}
\def\SO{\operatorname{SO}}

\def\bsf{\boldsymbol{f}}
\def\bsx{\boldsymbol{x}}
\def\bsy{\boldsymbol{y}}
\def\bsz{\boldsymbol{z}}

\def\fraka{\mathfrak{a}}

\def\frakp{\mathfrak{p}}
\def\frakq{\mathfrak{q}}

\def\calD{\mathcal{D}}

\def\CC{\mathbb{C}}
\def\NN{\mathbb{N}}
\def\QQ{\mathbb{Q}}
\def\ZZ{\mathbb{Z}}

\begin{document}
\title[Differential operators on invariant rings]{Differential operators on classical invariant rings do not lift modulo $p$}

\author{Jack Jeffries}
\address{Department of Mathematics, University of Nebraska, 203 Avery Hall, Lincoln, NE 68588, USA}
\email{jack.jeffries@unl.edu}

\author{Anurag K. Singh}
\address{Department of Mathematics, University of Utah, 155 South 1400 East, Salt Lake City, UT~84112, USA}
\email{singh@math.utah.edu}

\thanks{J.J.~was supported by NSF grants DMS~1606353 and DMS~2044833, and A.K.S.~by NSF grants DMS~1801285 and DMS~2101671.}

\begin{abstract}
Levasseur and Stafford described the rings of differential operators on various classical invariant rings of characteristic zero; in each of the cases that they considered, the differential operators form a simple ring. Towards an attack on the simplicity of rings of differential operators on invariant rings of linearly reductive groups over the complex numbers, Smith and Van den Bergh asked if differential operators on the corresponding rings of positive prime characteristic lift to characteristic zero differential operators. We prove that, in general, this is not the case for determinantal hypersurfaces, as well as for Pfaffian and symmetric determinantal hypersurfaces. We also prove that, with very few exceptions, these hypersurfaces---and, more generally, classical invariant rings---do not admit a mod $p^2$ lift of the Frobenius endomorphism.
\end{abstract}
\maketitle

\section{Introduction}
\label{section:intro}

For a polynomial ring $R$ over a field $K$ of characteristic zero, the ring of $K$-linear differential operators on $R$ is the $K$-algebra generated by $R$ and the $K$-linear derivations on~$R$, i.e., the ring $D_{R|K}=R\langle \frac{\partial}{\partial x_0},\dots,\frac{\partial}{\partial x_d}\rangle$. This noncommutative ring is the well-known Weyl algebra, which enjoys many good ring-theoretic properties: it is left and right Noetherian, and is a simple ring.

For commutative rings $R$ and $A$, where $R$ is an $A$-algebra, there is a notion, due to Grothendieck~\cite[\S16.8]{EGA4}, of the ring of $A$-linear differential operators on $R$, denoted $D_{R|A}$; see \S\ref{section:preliminaries}. However, in contrast with the case of a polynomial ring, if one takes $A$ to be a field~$K$ of characteristic zero, and $R$ to be $K[x,y,z]/(x^3+y^3+z^3)$, then $D_{R|K}$ is not left or right Noetherian, nor a finitely generated $K$-algebra, nor a simple ring; see~\cite{BGG}.

On the other hand, when $R$ is the ring of invariants for a linear action of a reductive group on a polynomial ring over a field $K$ of characteristic zero, it is known in many cases that~$D_{R|K}$ is Noetherian, finitely generated, and a simple ring, just as in the polynomial case \cite{Kantor1,LS,MV,Schwarz:ASENS}. Indeed, it is conjectured that for such invariant rings~$R$, the ring of differential operators~$D_{R|K}$ is a simple ring \cite{Schwarz:ICM}. An analogous statement in positive characteristic was proved by Smith and Van den Bergh \cite[Theorem~1.3]{SV}.

Quite generally, for $A$-algebras $R$ and $B$, there is an injective ring homomorphism
\[
D_{R\,|\,A} \otimes_A B\ \to\ D_{R\otimes_A B\,|\,B},
\]
that is an isomorphism when $B$ is flat over $A$. In particular, one has an isomorphism
\[
D_{R\,|\,\ZZ}\otimes_{\ZZ}\QQ\ \cong\ D_{R\otimes_{\ZZ}\QQ\,|\,\QQ},
\]
and, for $p$ a prime integer, an injective homomorphism
\begin{equation}
\label{equation:base:change}
D_{R\,|\,\ZZ}\otimes_\ZZ(\ZZ/p\ZZ)\ \into\ D_{(R/pR)\,|\,(\ZZ/p\ZZ)}.
\end{equation}

In order to relate rings of differential operators in characteristic zero to their counterparts in positive characteristic $p$, one needs to determine whether the map~\eqref{equation:base:change} is an isomorphism, i.e., whether each differential operator on $R/pR$ lifts to a differential operator on $R$. To study the problem of the simplicity of rings of differential operators on characteristic zero invariant rings via reduction to positive characteristic, Smith and Van den Bergh pose the following question, formulated here in equivalent terms:

\begin{question}[{\cite[Question~5.1.2]{SV}}]
\label{question:SV}
Let $A$ be a domain that is finitely generated as an algebra over $\ZZ$. Suppose $R$ is a finitely generated $A$-algebra such that $R\otimes_A\ffield(A)$ is the ring of invariants for a linear action of a reductive group on a polynomial ring of characteristic zero. Does there exist a nonempty open subset of $U$ of $\Spec A$, such that for each maximal ideal~$\mu\in U$, each differential operator on $D_{(R/\mu R)\,|\,(A/\mu A)}$ lifts to $D_{R|A}$?
\end{question}

We prove that the answer to Question~\ref{question:SV} is negative: for several classical invariant rings that are hypersurfaces, we construct explicit differential operators, modulo prime integers~$p$, that do not lift to characteristic zero differential operators. Our main theorem is below; we refer the reader to \S\ref{subsection:frobenius:trace} for the definition of Frobenius trace.

\begin{theorem}
\label{theorem:main}
Consider the following classical invariant rings:
\begin{enumerate}[\,\rm(a)]
\item\label{main:a} Let $X$ be an $n\times n$ matrix of indeterminates over $\ZZ$, with $n\ge 3$. Set $R\colonequals\ZZ[X]/(\det X)$. Then, for each prime integer $p>0$, the Frobenius trace on $R/pR$ does not lift to a differential operator on $R/p^2 R$, nor, a fortiori, to a differential operator on $R$.

\item\label{main:b} Let $X$ be an $n\times n$ alternating matrix of indeterminates over $\ZZ$, for $n\ge 4$ an even integer. Set~$R\colonequals\ZZ[X]/(\pf X)$, where $\pf X$ denotes the Pfaffian of $X$. Then, for each prime integer $p>0$, the Frobenius trace on $R/pR$ does not lift to a differential operator on~$R/p^2 R$, nor, a fortiori, to a differential operator on $R$.

\item\label{main:c} Let $X$ be a $3\times 3$ symmetric matrix of indeterminates over $\ZZ$. Set~$R\colonequals\ZZ[X]/(\det X)$. Then, for prime integers $p>2$, each differential operator on $R/pR$ lifts to a differential operator on $R$. In the case of characteristic $2$, the Frobenius trace on $R/2R$ does not lift to a differential operator on $R/4R$, nor, a fortiori, to a differential operator on $R$.
\end{enumerate}
\end{theorem}

For $R$ as in \eqref{main:a}, \eqref{main:b}, or \eqref{main:c} above, the ring $R\otimes_\ZZ\QQ$ is the invariant ring for an action of the linearly reductive group $\GL_{n-1}(\QQ)$, $\Sp_{2n-2}(\QQ)$, or $\O_2(\QQ)$, respectively; see, for example,~\S\ref{section:generic:determinant}, \S\ref{section:pfaffian}, \S\ref{section:symmetric}, for details. In contrast with the cases discussed above, we prove that if $R$ is a toric $\ZZ$-algebra, then, for each prime integer~$p>0$, every differential operator on $R/pR$ lifts to a differential operator on $R$; see Theorem~\ref{theorem:toric}. In particular, if $R$ is the hypersurface over $\ZZ$ defined by the determinant of a~$2\times 2$ matrix of indeterminates or a symmetric $2\times 2$ matrix of indeterminates, then every differential operator on $R/pR$ lifts to a differential operator on $R$; this addresses the case~$n=2$ in the context of Theorem~\ref{theorem:main}~\eqref{main:a} and \eqref{main:c} above.

Our approach is based on the paper \cite{Jeffries} by the first author, where it was established that there is an isomorphism between rings of differential operators (considered as modules over the enveloping algebra) and certain local cohomology modules; see also \cite{BZN} where related isomorphisms are established under different hypotheses. The isomorphism between rings of differential operators and local cohomology modules identifies differential operators that do not lift modulo a prime integer $p$ with local cohomology elements that do not lift modulo $p$, and consequently with nonzero elements in a different local cohomology module that are annihilated by the prime integer $p$. These ideas were used by the first author to give positive answers to Question~\ref{question:SV} in special cases~\cite[Theorem~6.3]{Jeffries}. However, we have attempted to keep the present paper largely self-contained, and focused on hypersurfaces, where the isomorphisms can be made entirely explicit; it is striking that the isomorphism is one of $D$-modules:

\begin{theorem}
\label{theorem:cyclic}
Let $A$ be a commutative ring. Set $R\colonequals A[x_0,\dots,x_n]/(f)$, where $f$ is a non\-zerodivisor in the polynomial ring $A[x_0,\dots,x_n]$. Set $\Delta$ to be the kernel of the multiplication map $R\otimes_A R\to R$. Then the local cohomology module $H^n_\Delta(R\otimes_A R)$, with the natural~$D_{R|A}$-module structure as in~\S\ref{subsection:equivalence}, is a free $D_{R|A}$-module of rank one.
\end{theorem}

The techniques and calculations used in our proof of Theorem~\ref{theorem:main} have implications to the existence of liftings of the Frobenius morphism that have attracted a lot of attention; for example, \cite{Bhatt:doc} provides a connection between liftability of the Frobenius and infinitely generated crystalline cohomology, while \cite{BTLM} proves Bott vanishing for varieties that admit a Frobenius lift modulo $p^2$. The liftability of the Frobenius is also studied in considerable detail in \cite{Zdanowicz}; we use results from that paper to determine whether the Frobenius endomorphism on a classical invariant ring of positive prime characteristic lifts to a ring endomorphism in characteristic zero:

\begin{theorem}
\label{theorem:frob}
Consider the following classical invariant rings, modeled over $\ZZ$:
\begin{enumerate}[\,\rm(a)]
\item\label{frobenius:a} $R\colonequals\ZZ[X]/I_t(X)$, where $X$ is an $m\times n$ matrix of indeterminates, $I_t(X)$ the ideal generated by the size $t$ minors of $X$, and $\min\{m,n\}\ge t\ge 3$.

\item\label{frobenius:b} $R\colonequals\ZZ[X]/\Pf_t(X)$, where $X$ is an $n\times n$ alternating matrix of indeterminates, $\Pf_t(X)$ the ideal generated by the Pfaffians of the size~$t$ principal submatrices of $X$, for $t$ even, and $n\ge t\ge 4$.

\item\label{frobenius:c} $R\colonequals\ZZ[X]/I_t(X)$, where $X$ is a symmetric $n\times n$ matrix of indeterminates, $I_t(X)$ the ideal generated by the size $t$ minors of $X$, and $n\ge t\ge 3$.
\end{enumerate}
Let $p$ be a positive prime integer. In cases~\eqref{frobenius:a}~and~\eqref{frobenius:b}, the Frobenius endomorphism on~$R/pR$ does not lift to a ring endomorphism of $R/p^2R$, nor, a fortiori, to a ring endomorphism of~$R$. In case~\eqref{frobenius:c}, the same conclusion holds if $t\ge 4$ or if $p=2$.
\end{theorem}

If $R$ is a normal affine semigroup ring over $\ZZ$, then, for each prime integer~$p>0$, the Frobenius endomorphism on $R/pR$ lifts to an endomorphism of $R$, and hence to an endomorphism of $R/p^2R$. Specifically, if $R$ is defined by the size~$2$ minors of a matrix of indeterminates, or of a symmetric matrix of indeterminates, then the Frobenius endomorphism on~$R/pR$ lifts to an endomorphism of $R$ and of $R/p^2R$; this explains the case~$t=2$ in the context of Theorem~\ref{theorem:frob}~\eqref{frobenius:a} and \eqref{frobenius:c}.

Recall that if $G$ is a linearly reductive group over a field $K$, with a linear action on a polynomial ring $K[\bsx]$, then the invariant ring $K[\bsx]^G$ is a direct summand of $K[\bsx]$ as a~$K[\bsx]^G$-module; many key properties of classical invariant rings including finite generation and the Cohen-Macaulay property follow from the existence of such a splitting, see for example~\cite[\S2]{Hochster-Eagon} and \cite{Hochster-Roberts}. Indeed, the general linear group, the symplectic group, and the orthogonal group are linearly reductive over fields of characteristic zero; it follows that determinantal rings, Pfaffian determinantal rings, and symmetric determinantal rings, over fields of characteristic zero, are direct summands of polynomial rings. In contrast, we prove that working over the ring of integers $\ZZ$, or over the ring of $p$-adic integers $\widehat{\ZZ_{(p)}}$, the corresponding rings are typically not direct summands of \emph{any} polynomial ring:

\begin{corollary}
\label{corollary:not:summand}
Let $R$ be as in Theorem~\ref{theorem:frob}~\eqref{frobenius:a},~\eqref{frobenius:b}, or~\eqref{frobenius:c}. Then $R$ is not a direct summand, as an~$R$-module, of any polynomial ring over $\ZZ$.

Let $V\colonequals\widehat{\ZZ_{(p)}}$ be the $p$-adic integers; in case~\eqref{frobenius:c}, assume further that either~$t\ge 4$, or that $p=2$. Then the ring $R\otimes_\ZZ V$ is not a direct summand, as an $R\otimes_\ZZ V$-module, of any polynomial ring over $V$.
\end{corollary}

The corollary is immediate from Theorem~\ref{theorem:frob} since the existence of a Frobenius lift modulo $p^2$ is inherited by a ring that is a direct summand, see Proposition~\ref{proposition:lift:summand}. Regarding the exceptional case~\eqref{frobenius:c} in Theorem~\ref{theorem:frob} and the corollary, if~$V$ is a discrete valuation ring of mixed characteristic, such that the residual characteristic is an odd prime integer, then a symmetric determinantal ring of the form $V[X]/I_3(X)$ is a direct summand of a polynomial ring over $V$, see Remark~\ref{remark:dirsum} and \S\ref{section:minors}.

In \S\ref{section:preliminaries} we record basic facts about differential operators and Koszul and local cohomology modules from perspectives needed later in the paper; \S\ref{section:isomorphism} includes the aforementioned explicit isomorphisms, relating specific $p$-torsion local cohomology elements to specific differential operators; Theorem~\ref{theorem:cyclic} is proved in \S\ref{subsection:equivalence}. The proofs of the three cases of Theorems~\ref{theorem:main} and~\ref{theorem:frob} are completed in \S\ref{section:pfaffian}, \S\ref{section:generic:determinant},~and \S\ref{section:symmetric}, while Theorem~\ref{theorem:frob} is proved in \S\ref{section:minors}. For Theorem~\ref{theorem:main}, we prove that the relevant $p$-torsion local cohomology elements are nonzero. We expect that these calculations are of independent interest, adding to the study of integer torsion in local cohomology modules pursued in \cite{Singh:MRL,LSW,BBLSZ}.

\section{Preliminaries}
\label{section:preliminaries}

\subsection{Differential operators}
\label{subsection:differential operators}

\emph{Differential operators} on a commutative ring $R$ are defined inductively as follows: for each $r\in R$, the multiplication by $r$ map $\tilde{r}\colon R\to R$ is a differential operator of order $0$; for each positive integer $n$, the differential operators of order less than or equal to $n$ are additive maps $\delta\colon R\to R$ for which each commutator
\[
[\tilde{r},\delta]\colonequals\tilde{r}\circ\delta-\delta\circ\tilde{r}
\]
is a differential operator of order less than or equal to $n-1$. If $\delta$ and $\delta'$ are differential operators of order at most $m$ and $n$ respectively, then $\delta\circ\delta'$ is a differential operator of order at most $m+n$. Thus, the differential operators on~$R$ form a subring $D_R$ of $\End_\ZZ(R)$. We use~$D^n_R$ to denote the differential operators of order at most $n$.

A differential operator~$\delta\in D_R$ is a \emph{derivation} if
\[
\delta(r_1r_2)=r_1\delta(r_2)+r_2\delta(r_1)\qquad\text{for }\ r_i\in R.
\]
It is readily seen that each element of $D^1_R$ may be expressed uniquely as the sum of an element of $D^0_R$ and a derivation.

When $R$ is an algebra over a commutative ring $A$, we set $D_{R|A}$ to be the subring of~$D_R$ that consists of differential operators that are $A$-linear; note that $D_{R|\ZZ}$ equals $D_R$. When~$R$ is an algebra over a perfect field $K$ of positive prime characteristic, then $D_{R|K}$ equals $D_R$, see for example \cite[Example~5.1~(c)]{Lyubeznik:Crelle}.

If $R=A[x_0,\dots,x_d]$ is a polynomial ring over $A$, then any element of $D^n_{R|A}$ can be expressed as an $R$-linear combination of the differential operators $\partial_{a_0,\dots,a_d}$, where $a_i$ are nonnegative integers with $a_0+\cdots+a_d\le n$, and
\[
\partial_{a_0,\dots,a_d} (x_0^{b_0} \cdots x_d^{b_d})\ =\ \prod_{i=0}^d \binom{b_i}{a_i} x_0^{b_0-a_0} \cdots x_d^{b_d-a_d}.
\]
Note that if $A$ contains the field of rational numbers, then
\[
\partial_{a_0,\dots,a_d}\ =\ \frac{1}{a_0!}\frac{\partial^{a_0}}{\partial x_0^{a_0}} \cdots \frac{1}{a_d!}\frac{\partial^{a_d}}{\partial x_d^{a_d}}.
\]

We record an alternative description of $D_{R|A}$ from~\cite[\S 16.8]{EGA4}. For $R$ an $A$-algebra, set
\[
P_{R\,|\,A}\colonequals R\otimes_A R,
\]
and consider the $P_{R|A}$-module structure on $\End_A(R)$ under which $r_1\otimes r_2$ acts on $\delta$ to give the endomorphism $\tilde{r}_1\circ\delta\circ\tilde{r}_2$, where $\tilde{r}_i$ denotes the map that is multiplication by $r_i$. Set
\[
\CD
\Delta_{R\,|\,A}\colonequals\ker(P_{R\,|\,A}@>\mu>> R),
\endCD
\]
where $\mu$ is the $A$-algebra homomorphism determined by $\mu(r_1\otimes r_2)=r_1r_2$. The ideal $\Delta_{R|A}$ is generated by elements of the form $r\otimes\!1-1\!\otimes r$. Since
\[
(r\otimes\!1-1\!\otimes r)(\delta)\ =\ [\tilde{r},\delta],
\]
it follows that an element $\delta$ of $\End_A(R)$ is a differential operator of order at most $n$ precisely if it is annihilated by $\Delta^{n+1}_{R|A}$. By~\cite[Proposition~16.8.8]{EGA4}, the $A$-linear differential operators on $R$ of order at most~$n$ correspond to
\begin{equation}
\label{equation:diff:iso}
\Hom_R(P^n_{R\,|\,A},\ R)\ \cong\ \ann(\Delta^{n+1}_{R\,|\,A},\ \End_A(R)),
\end{equation}
where
\[
P^n_{R\,|\,A}\colonequals P_{R\,|\,A}/\Delta^{n+1}_{R\,|\,A}
\]
is viewed as a left $R$-module via $r\mapsto r\otimes\!1$. To make the isomorphism~\eqref{equation:diff:iso} explicit, consider the map
\[
\rho\colon R\to P^n_{R\,|\,A}\qquad\text{with }\ r\mapsto 1\!\otimes r,
\]
in which case, the map
\[
\rho^*\colon\Hom_R(P^n_{R\,|\,A},\ R)\to D^n_{R\,|\,A}\qquad\text{with }\ \delta\mapsto\delta\circ\rho
\]
is an isomorphism of $P_{R|A}$-modules.

In addition to the filtration by order, we will have use for another filtration on $D_{R|A}$. Supposing that $\Delta_{R|A}$ is a finitely generated ideal of $P_{R|A}$, fix a set of generators $z_0,\dots,z_t$. Then the sequence of ideals defined by
\[
\Delta^{[n]}_{R\,|\,A}\colonequals(z_0^n,\dots,z_t^n),\qquad\text{for }\ n\in\NN,
\]
is cofinal with the sequence $\Delta_{R|A}^n$ for $n\in\NN$. As there is little risk of confusion, we reuse the notation $\rho\colon R\to P^{[n]}_{R|A}$ for the map with $r\mapsto1\!\otimes r$, and set
\[
D^{[n]}_{R\,|\,A}\colonequals\{\delta\circ\rho\ \mid\ \delta\in\Hom_R(P^{[n]}_{R\,|\,A},\ R)\}, \qquad\text{where }\ 
P^{[n]}_{R\,|\,A}\colonequals P_{R\,|\,A}/\Delta^{[n+1]}_{R\,|\,A}.
\]
Similarly, we use $\rho^*$ for the $P_{R|A}$-module isomorphism
\[
\rho^*\colon\Hom_R(P^{[n]}_{R\,|\,A},\ R)\to D^{[n]}_{R\,|\,A}\qquad\text{with }\ \delta\mapsto\delta\circ\rho.
\]
Note that $D^{[n]}_{R|A}$ gives a filtration of $D_{R|A}$ that is cofinal with the filtration by order; the filtration $D^{[n]}_{R|A}$ depends on the choice of ideal generators for $\Delta_{R|A}$.

Let $M$ be an $R$-module. Then $\Hom_A(R,M)$ has a $P_{R|A}$-module structure given by
\[
(r_1\otimes r_2)\cdot\delta\colonequals\tilde{r}_1\circ\delta\circ\tilde{r}_2.
\]
The differential operators from $R$ to $M$, of order at most $n$, are 
\[
D^n_{R\,|\,A}(M)\colonequals\ann\big(\Delta^{n+1}_{R\,|\,A},\ \Hom_A(R,M)\big).
\]
Note that one has an isomorphism of $P_{R|A}$-modules,
\[
D^n_{R\,|\,A}(M)\ \cong\ \Hom_R(P^n_{R\,|\,A},\ M).
\]
Likewise, given generators $z_0,\dots,z_t$ for $\Delta_{R|A}$, we set
\[
D^{[n]}_{R\,|\,A}(M)\colonequals\ann\big(\Delta^{[n+1]}_{R\,|\,A},\ \Hom_A(R,M)\big)\ \cong\ \Hom_R(P^{[n]}_{R\,|\,A},\ M)
\]
and
\[
D_{R\,|\,A}(M)\colonequals\bigcup_{n\ge0} D^n_{R\,|\,A}(M).
\]

In analogy with the equality $D_R = D_{R|\ZZ}$, we set $P_R\colonequals P_{R|\ZZ}$ and $\Delta_R\colonequals \Delta_{R|\ZZ}$, along with the corresponding notation $P^n_R\colonequals P_R/\Delta^{n+1}_R$ and $P^{[n]}_R\colonequals P_R/\Delta^{[n+1]}_R$. Observe that if $R$ has characteristic $p$, we have $P_R=P_{R|(\ZZ/p\ZZ)}$, and likewise for the other notions discussed above.

\subsection{Koszul and local cohomology}
\label{subsection:koszul:local}

Let $z$ be an element of a ring $R$. One has maps between the Koszul complexes $K^\bullet(z^n;\,R)$, for $n\in\NN$, and the \v Cech complex $C^\bullet(z;\,R)$ as below:
\[
\CD
0 @>>> R @>z^{n-1}>> R @>>> 0\phantom{.}\\
@. @V1VV @VVzV\\ 
0 @>>> R @>z^n>> R @>>> 0\phantom{.}\\
@. @V1VV @VV1/z^nV\\ 
0 @>>> R @>>> R_z @>>> 0.
\endCD
\]
Taking the direct limit of $K^\bullet(z^n;\,R)$ yields an isomorphism $\varinjlim_n K^\bullet(z^n;\,R)\cong C^\bullet(z;\,R)$.

For $\bsz\colonequals z_0,\dots,z_t$, one similarly obtains a map of complexes
\[
\CD
K^\bullet(\bsz;\ R)\colonequals\bigotimes_i K^\bullet(z_i;\ R) @>>> \bigotimes_i C^\bullet(z_i;\ R)\equalscolon C^\bullet(\bsz;\ R),
\endCD
\]
and, for each $k\ge 0$, an induced map from Koszul cohomology to local cohomology
\begin{equation}
\label{equation:koszul:local}
\CD
H^k(\bsz;\ R) @>>> H^k_{(\bsz)}(R).
\endCD
\end{equation}
Setting $\bsz^n\colonequals z_0^n,\dots,z_t^n$, one likewise has $\varinjlim_n K^\bullet(\bsz^n;\,R)\cong C^\bullet(\bsz;\,R)$, and 
\[
\varinjlim_n H^k(\bsz^n;\ R)\ \cong\ H^k_{(\bsz)}(R)\qquad\text{for each }\ k\ge0 .
\]
When $k$ equals $t+1$, the map~\eqref{equation:koszul:local} takes the form
\[
H^{t+1}(\bsz;\ R)=\frac{R}{(\bsz)R}\ \to\ \frac{R_{z_0\cdots z_t}}{\sum_i R_{z_0\cdots \hat{z}_i\cdots z_t}}=H^{t+1}_{(\bsz)}(R),
\qquad\text{with }\ \bar{1}\ \mapsto\ \left[\frac{1}{z_0\cdots z_t}\right],
\]
while, if $k$ equals $t$, the Koszul cohomology element in $H^t(\bsz;\,R)$ corresponding to an equation $\sum_iz_ig_i=0$ in $R$ maps to
\[
\left[\cdots,\ \frac{(-1)^ig_i}{z_0\cdots \hat{z}_i\cdots z_t},\ \cdots\right]\ \in\ H^t_{(\bsz)}(R).
\]

We mention that a local cohomology element
\[
\left[\frac{r}{z_0^n\cdots z_t^n}\right]\ \in\ H^{t+1}_{(\bsz)}(R)
\]
is zero if and only if there exists an integer $k\ge0$ such that
\[
r(z_0\cdots z_t)^k\ \in\ \big(z_0^{n+k},\ \dots,\ z_t^{n+k} \big)R.
\]

Suppose the ring $R$ takes the form $S/fS$, for $S$ a commutative ring, and $f\in S$ a regular element. Let $\bsz\colonequals z_0,\dots,z_t$, as before. The exact sequence
\[
\CD
0 @>>> S @>f>> S @>>> R @>>> 0
\endCD
\]
induces the cohomology exact sequence
\[
\CD
@>>> H^t_{(\bsz)}(R) @>\delta_f>> H^{t+1}_{(\bsz)}(S) @>f>> H^{t+1}_{(\bsz)}(S) @>>>,
\endCD
\]
with $\delta_f$ denoting the connecting homomorphism. To make the map explicit, suppose $g_i$ are elements of $S$ such that
\[
\sum_iz_i^ng_i= sf
\]
for some $n\ge 1$ and an element $s\in S$; then
\[
\delta_f\colon\left[\cdots,\ \frac{(-1)^ig_i}{(z_0\cdots \hat{z}_i\cdots z_t)^n},\ \cdots\right]\ \mapsto\ \left[\frac{s}{(z_0\cdots z_t)^n}\right].
\]

\subsection{Bockstein homomorphisms}
\label{subsection:bockstein}

We briefly review Bockstein homomorphisms on local cohomology \cite{SinghWalther:Bock}. Let $p$ be a prime integer that is a regular element on a ring $R$. Fix an ideal $\fraka$ of $R$. Applying the local cohomology functor $H^\bullet_\fraka(-)$ to
\[
\CD
0 @>>> R/pR @>p>> R/p^2R @>>> R/pR @>>> 0,
\endCD
\]
one obtains a cohomology exact sequence; the \emph{Bockstein homomorphism}
\[
\beta_p\colon H^k_\fraka(R/pR)\to H^{k+1}_\fraka(R/pR)
\]
is the connecting homomorphism in the cohomology exact sequence. For an alternative point of view, one may take the cohomology exact sequence induced by
\[
\CD
0 @>>> R @>p>> R @>>> R/pR @>>> 0,
\endCD
\]
i.e., the sequence
\[
\CD
@>>> H^k_\fraka(R/pR) @>{\delta_p}>> H^{k+1}_\fraka(R) @>p>> H^{k+1}_\fraka(R) @>{\pi_p}>> H^{k+1}_\fraka(R/pR) @>>>.
\endCD
\]
The Bockstein homomorphism $\beta_p$ above then coincides with the composition
\[
\CD
H^k_\fraka(R/pR) @>{\ \ \pi_p\, \circ \, \delta_p\ \ }>> H^{k+1}_\fraka(R/pR).
\endCD
\]

\subsection{The \texorpdfstring{$D$}{D}-module structure on local cohomology}
\label{subsection:d:mod:local:cohomology}

Let $R$ be an $A$-algebra, and $z$ an element of $R$. Then the $D_{R|A}$-module structure on $R$ extends uniquely to the localization $R_z$ as follows: one defines the action by induction on the order of $\delta\in D_{R|A}$ and on the power of the denominator by the rule
\[
\delta(r/z^n)\colonequals \frac{\delta(r/z^{n-1}) - [\delta,z](r/z^n)}{z}.
\]
This may be written in closed form: set $\delta^{(0)}\colonequals \delta$ and $\delta^{(i+1)}\colonequals [\delta^{(i)},\tilde{z^n}]$ inductively;~then
\begin{equation}
\label{equation:localization2}
\delta(r/z^n)= \sum_{k=0}^{\mathrm{ord}(\delta)} (-1)^k \ \frac{\delta^{(k)}(r)}{z^{n(k+1)}}.
\end{equation}
For elements $\bsz$ of $R$, the \v Cech complex $C^\bullet(\bsz;\,R)$ is a complex of $D_{R|A}$-modules, hence its cohomology modules $H^k_{(\bsz)}(R)$ have a natural $D_{R|A}$-module structure. This $D_{R|A}$-module structure on local cohomology is compatible with base change in the following sense: If~$S\to R$ is a homomorphism of $A$-algebras, we observe that $D_{S|A}(R)$ has the structure of a $(D_{R|A},D_{S|A})$-bimodule, where $D_{R|A}$ acts by postcomposition, and $D_{S|A}$ acts by precomposition; one verifies using the inductive definition of differential operators that the prescribed compositions are indeed elements of $D_{S|A}(R)$. This bimodule structure yields a base change functor from $D_{S|A}$-modules to $D_{R|A}$-modules, $M \mapsto D_{S|A}(R)\otimes_{D_{S|A}} M$.

\begin{lemma}
\label{lemma:bimodule}
Let $S=A[\bsx]$ be a polynomial ring over $A$, and $R$ a homomorphic image of $S$.
\begin{enumerate}[\,\rm(a)]
\item\label{lemma:bimodule:a} For each $z\in S$, one has an isomorphism of $D_{R|A}$-modules $D_{S|A}(R)\otimes_{D_{S|A}} S_z \cong R_z$.
\item\label{lemma:bimodule:b} Given elements $\bsz\colonequals z_0,\dots,z_t$ of $S$, one has an isomorphism of $D_{R|A}$-modules
\[
D_{S\,|\,A}(R)\otimes_{D_{S\,|\,A}} H^{t+1}_{(\bsz)}(S) \cong H^{t+1}_{(\bsz)}(R).
\]
\end{enumerate}
\end{lemma}

\begin{proof}
Set $I\colonequals\ker(S\to R)$. Since $S$ is a polynomial ring over $A$, the functor $D_{S|A}(-)$ from $S$-modules to~$D_{S|A}$-modules is exact~\cite[\S2.2]{SV}, so there is an exact sequence of right~$D_{S|A}$-modules
\[
\CD
0 @>>> D_{S\,|\,A}(I) @>>> D_{S\,|\,A} @>>> D_{S\,|\,A}(R) @>>> 0.
\endCD
\]
The inclusion above may be used to identify $D_{S|A}(I)$ with $I D_{S|A}$, so $D_{S|A}(R)\cong D_{S|A}/ID_{S|A}$ as right $D_{S|A}$-modules. Thus, one has an isomorphism of $R$-modules
\[
D_{S\,|\,A}(R)\otimes_{D_{S\,|\,A}} S_z\ \cong\ S_z/I S_z\ \cong\ R_z,
\qquad\text{given by} \quad
\bar{\delta}\otimes (s/z^n) \mapsto \bar{\delta(s/z^n)},
\]
where $\bar{\delta}$ denotes the image of $\delta\in D_{S|A}$ modulo $ID_{S|A}$. To see that this isomorphism is~$D_{R|A}$-linear, note that each element of $D_{S|A}(R)\otimes_{D_{S|A}}S_z$ may be written as $\bar{1}\otimes s/z^n$, and any element $\gamma\in D_{R|A}$ can be written as $\mu + I D_{S|A}$ for some $\mu\in D_{S|A}$ such that $\mu(I)\subseteq I$. The compatibility condition then boils down to checking that $\gamma(\bar{s/z^n})=\bar{\mu(s/z^n)}$, which is clear from the formula~\eqref{equation:localization2}, proving~\eqref{lemma:bimodule:a}.

Next, consider the right-exact sequence of $D_{S|A}$-modules
\[
\CD
\sum_i S_{z_0 \cdots \hat{z}_i \cdots z_t} @>>> S_{z_0\cdots z_t} @>>> H^{t+1}_{(\bsz)}(S) @>>> 0,
\endCD
\]
and the right-exact sequence of $D_{R|A}$-modules
\[
\CD
\sum_i R_{z_0 \cdots \hat{z}_i \cdots z_t} @>>> R_{z_0\cdots z_t} @>>> H^{t+1}_{(\bsz)}(R) @>>> 0.
\endCD
\]
Applying $D_{S|A}(R)\otimes_{D_{S|A}}(-)$ to the first, we obtain a commutative diagram of $D_{R|A}$-modules 
\[
\minCDarrowwidth16pt
\CD
D_{S|A}(R) \otimes_{D_{S|A}} \sum_i S_{z_0 \cdots \hat{z}_i \cdots z_t} @>>> D_{S|A}(R) \otimes_{D_{S|A}} S_{z_0\cdots z_t} @>>> D_{S|A}(R) \otimes_{D_{S|A}} H^{t+1}_{(\bsz)}(S) @>>> 0 \\
@VVV @VVV @. @.\\
\sum_i R_{z_0 \cdots \hat{z}_i \cdots z_t} @>>> R_{z_0\cdots z_t} @>>> H^{t+1}_{(\bsz)}(R) @>>> 0,
\endCD
\]
where the vertical maps are isomorphisms by~\eqref{lemma:bimodule:a}. This induces the $D_{R|A}$-module isomorphism as claimed in~\eqref{lemma:bimodule:b},
\end{proof}

\subsection{Frobenius lifting and \texorpdfstring{$p$}{p}-derivations}
\label{subsection:FrobeniusLifting}

Let $T$ be a ring, and let $p>0$ be a prime integer that is not a unit in $T$. A \emph{lift} of the Frobenius endomorphism $F$ on $T/pT$ is a ring homomorphism $\Lambda_p\colon T\to T$ such that the following diagram commutes
\[
\CD
T @>\Lambda_p>> T \\
@VVV @VVV \\
T/pT @>F>> T/pT.
\endCD
\]

As we show next, the existence of a Frobenius lift modulo $p^2$ is inherited by rings that are direct summands; see also~\cite[Lemma~4.1]{Zdanowicz}.

\begin{proposition}
\label{proposition:lift:summand}
Let $T$ be a ring, and let $p>0$ be a prime integer that is not a unit in $T$. Let $R$ be a subring of $T$ that is a direct summand of $T$ as an $R$-module.

If the Frobenius endomorphism on $T/pT$ lifts to an endomorphism of $T/p^2T$, then the Frobenius endomorphism on $R/pR$ lifts to an endomorphism of $R/p^2R$.
\end{proposition}

\begin{proof}
Note that $R/p^2R$ is a direct summand of $T/p^2T$; after a change of notation, we may assume that $p^2=0$ in $R$ and $T$. We use $\iota\colon R\to T$ for the inclusion, and $\rho\colon T\to R$ to denote an $R$-linear splitting.

Let $\Lambda_p\colon T\to T$ denote a Frobenius lift; we claim that
\[
\rho\circ\Lambda_p\circ\iota\colon R\to R
\]
is an endomorphism of $R$. Since each map is additive, so is the composition. Given elements $r_i$ in $R$, there exist elements $t_i$ in $T$ such that $\Lambda_p\circ\iota(r_i) = r_i^p+pt_i$. Hence
\begin{alignat*}2
\rho\circ\Lambda_p\circ\iota(r_1r_2) &=\ \rho\big((r_1^p+pt_1)(r_2^p+pt_2)\big)\\
&=\ r_1^pr_2^p + r_1^p p\rho(t_2) + r_2^p p\rho(t_1)\\
&=\ \big(r_1^p + p\rho(t_1)\big)\big(r_2^p + p\rho(t_2)\big)\\
&=\ \big(\rho\circ\Lambda_p\circ\iota(r_1)\big)\big(\rho\circ\Lambda_p\circ\iota(r_2)).
\end{alignat*}
It is readily verified that $\rho\circ\Lambda_p\circ\iota$ induces the Frobenius endomorphism on $R/pR$.
\end{proof}

\begin{definition}[Buium \cite{Buium}, Joyal \cite{Joyal}]
Let $T$ be a ring, and $p>0$ be a prime integer. A~\emph{p-derivation} on $T$ is a map $\phi_p\colon T\to T$ that satisfies the following for all $a,b\in T$:
\begin{enumerate}[\,\rm(i)]
\item $\phi_p(1)=0$,
\item $\phi_p(ab)=a^p \phi_p(b) + b^p \phi_p(a) + p \phi_p(a) \phi_p(b)$, and
\item $\phi_p(a+b)=\phi_p(a) + \phi_p(b) + C_p(a,b)$, 
\end{enumerate}
where $C_p(x,y)$ is the polynomial $\frac{1}{p}(x^p+y^p-(x+y)^p)$ regarded as an element of $\ZZ[x,y]$. 
\end{definition}

It follows from the above that
\[
\phi_p(a+pb)\ \equiv\ \phi_p(a) + b^p \mod p.
\]

If $\phi_p$ is a $p$-derivation on $T$, then the map $\Lambda_p\colon T\to T$ given by $\Lambda_p(t)=t^p+p \phi_p(t)$ is a lift of the Frobenius; conversely, if $p$ is a nonzerodivisor on $T$, and $\Lambda_p\colon T\to T$ is a lift of the Frobenius, then $\phi_p\colon T\to T$ with $\phi_p(t)= \frac{1}{p}(\Lambda_p(t)-t^p)$ is a $p$-derivation on $T$.

\begin{example}
\label{example:standard:p:der}
Fix a prime integer $p>0$, and let $S$ be a polynomial ring over $\ZZ$ in the indeterminates $\bsx\colonequals x_0,\dots,x_d$. We refer to the $\ZZ$-algebra homomorphism
\[
\Lambda_p\colon S\to S\qquad\text{with }\ \Lambda_p(x_i)=x_i^p\ \text{ for each $i$}
\]
as the \emph{standard lift of the Frobenius} with respect to $\bsx$, and the corresponding $p$-derivation
\[
\phi_p\colon S\to S\qquad\text{with }\ \phi_p(s)=\frac{\Lambda_p(s)-s^p}{p}
\]
as the \emph{standard $p$-derivation} with respect to $\bsx$.
\end{example}

We record the following compatibility for $p$-derivations used in the sequel: Let $S$ and~$S'$ be polynomial rings over $\ZZ$ in the indeterminates $\bsx$ and $\bsx'$ respectively. Let~$\Upsilon\colon S\to S'$ be a ring homomorphism such that $\Upsilon(x)\in\bsx'\cup\{0,1\}$ for each $x\in\bsx$, i.e., the homomorphism~$\Upsilon$ either takes an indeterminate $x\in\bsx$ to an indeterminate $x'\in\bsx'$, or specializes it to $0$ or $1$. Let~$\Lambda_p$ and $\phi_p$ denote the standard lift of the Frobenius and the corresponding $p$-derivation on $S$ with respect to $\bsx$, and likewise let $\Lambda_p'$ and $\phi_p'$ be the corresponding maps for~$S'$ with respect to~$\bsx'$. Then the following diagrams commute:
\[
\CD
S @>\Upsilon>> S' @. \qquad\qquad\qquad S @>\Upsilon>> S'\\
@VV\Lambda_pV @VV{\Lambda_p'}V \qquad\qquad\qquad @VV\phi_pV @VV\phi_p'V\\
S @>\Upsilon>> S', @. \qquad\qquad\qquad S @>\Upsilon>> S'.
\endCD
\]

We recall the following criterion, due to Zdanowicz, for the existence of a lift of the Frobenius endomorphism on a hypersurface:

\begin{proposition}[{\cite[Corollary~4.9]{Zdanowicz}}]
\label{proposition:zdanowicz}
Let $p$ be a prime integer, $S\colonequals\ZZ[\bsx]$ a polynomial ring, and $\phi_p\colon S\to S$ the standard $p$-derivation on $S$ with respect to $\bsx\colonequals x_0,\dots,x_d$. For~$f\in S$, set~$R\colonequals S/fS$. Suppose that $R/p^2R$ is flat over $\ZZ/p^2 \ZZ$.

Then the Frobenius endomorphism on the hypersurface $R/pR$ lifts to an endomorphism of $R/p^2R$ if and only if
\[
\phi_p(f)\ \in\ \bigg(p,\ f,\ \Big(\frac{\partial f}{\partial x_0}\Big)^p,\ \dots,\ \Big(\frac{\partial f}{\partial x_d}\Big)^p \bigg)S.
\]
\end{proposition}

\section{Differential operators via local cohomology}
\label{section:isomorphism}

\subsection{Polynomial rings}
\label{subsection:polynomial}

Let $A$ be a commutative ring, and $S\colonequals A[\bsx]$ the polynomial ring over $A$ in the indeterminates $\bsx\colonequals x_0,\dots,x_d$. Then $P_{S|A}$ is a polynomial ring over $A$ in the indeterminates $x_i\otimes\!1$ and $1\!\otimes x_i$, for $0\le i \le d$. Since we view $P_{S|A}$ as an $S$-module via the map $s\mapsto s\otimes\!1$, we simply write $x_i$ for $x_i\otimes\!1$. Set $y_i\colonequals 1\!\otimes x_i$, so that
\[
P_{S\,|\,A}\ =\ A[\bsx,\bsy]\ =\ S[\bsy],
\]
where $\bsy\colonequals y_0,\dots,y_d$. The elements
\[
y_0-x_0,\ y_1-x_1,\ \dots,\ y_d-x_d
\]
form a generating set for the ideal $\Delta_{S|A}$ of $P_{S|A}$. Using this generating set, we consider the sequence of ideals $\Delta^{[n]}_{S|A}$, the rings $P^{[n]}_{S|A}$, and the filtration $D^{[n]}_{S|A}$, as defined in \S\ref{subsection:differential operators}.

Note that the elements $y_0-x_0,\ y_1-x_1,\ \dots,\ y_d-x_d$ are algebraically independent generators for $P_{S|A}$ as an $S$-algebra. Thus, $P^{[n]}_{S|A}$ is a free $S$-module with basis
\[
(y_0-x_0)^{a_0} \cdots (y_d-x_d)^{a_d}\qquad\text{where }\ 0\le a_i\le n.
\]
Likewise, $\Hom_S(P^{[n]}_{S|A},\,S)$ is the free $S$-module with the dual basis
\[
\big((y_0-x_0)^{a_0} \cdots (y_d-x_d)^{a_d}\big)^\star,\qquad\text{where }\ 0\le a_i\le n,
\]
and $(-)^\star$ denotes the corresponding element of the dual basis. Moreover, there is a $P_{S|A}$-module isomorphism defined $S$-linearly by the rule
\begin{align*}
\gamma_n\colon P^{[n]}_{S\,|\,A} &\ \to\ \Hom_S(P^{[n]}_{S\,|\,A},\ S) \\
(y_0-x_0)^{a_0}\cdots (y_d-x_d)^{a_d} &\ \mapsto\ \big((y_0-x_0)^{n-a_0} \cdots (y_d-x_d)^{n-a_d}\big)^\star.
\end{align*}

\begin{proposition}
\label{proposition:iso:koszul}
For each $n\ge0$, one has $P_{S|A}$-module isomorphisms
\[
\CD
H^{d+1}(\Delta^{[n+1]}_{S\,|\,A};\ P_{S\,|\,A}) @>\gamma_n>> \Hom_S(P^{[n]}_{S\,|\,A},\ S) @>\rho^*>> D^{[n]}_{S\,|\,A}.
\endCD
\]
\end{proposition}

\begin{proof}
We need only observe that the Koszul cohomology module
\[
H^{d+1}(\Delta^{[n+1]}_{S\,|\,A};\ P_{S\,|\,A})
\]
coincides with~$P^{[n]}_{S|A}$. The rest is immediate from the preceding discussion.
\end{proof}

To make the isomorphism $\rho^*$ completely explicit, note first that $\rho(f(\bsx))=f(\bsy)$. Now,
\begin{align*}
\rho^*\Big(\big((y_0-x_0)^{a_0} &\cdots (y_d-x_d)^{a_d}\big)^\star\Big) \big(f(\bsx)\big) \\
&=\ \big((y_0-x_0)^{a_1} \cdots (y_d-x_d)^{a_d}\big)^\star \big(f(\bsy)\big) \\ 
&=\ \big((y_0-x_0)^{a_0} \cdots (y_d-x_d)^{a_d}\big)^\star \Big(f\big((y_0-x_0)+x_0,\ \dots,\ (y_d-x_d)+x_d\big)\Big)\\ 
&=\ \partial_{a_0,\dots,a_d}\big(f(\bsx)\big),
\end{align*}
where the last equality uses the Taylor expansion of a polynomial. It follows that
\[
\rho^*\Big(\big((y_0-x_0)^{a_0} \cdots (y_d-x_d)^{a_d}\big)^\star\Big)\ =\ \partial_{a_0,\dots,a_d}.
\]

\begin{proposition}
\label{proposition:iso:local:S}
There is a $P_{S|A}$-module isomorphism
\[
\CD
H^{d+1}_{\Delta_{S\,|\,A}}(P_{S\,|\,A}) @>{\ \ \rho^*\,\circ\,\gamma\ \ }>> D_{S\,|\,A},
\endCD
\]
where $\gamma$ is the isomorphism
\[
H^{d+1}_{\Delta_{S\,|\,A}}(P_{S\,|\,A})\ =\ \varinjlim_n H^{d+1}_{\Delta_{S\,|\,A}}(P_{S\,|\,A})\to\varinjlim_n\Hom_S(P^{[n]}_{S\,|\,A},\ S).
\]
\end{proposition}

\begin{proof}
For each $n\ge0$, one has a commutative diagram
\[
\CD
H^{d+1}(\Delta^{[n+1]}_{S\,|\,A};\ P_{S\,|\,A}) @>\gamma_n>> \Hom_S(P^{[n]}_{S\,|\,A},\ S) @>\rho^*>> D^{[n]}_{S\,|\,A}\\
@VV\prod_i (y_i-x_i)V @VVV @VVV\\
H^{d+1}(\Delta^{[n+2]}_{S\,|\,A};\ P_{S\,|\,A}) @>\gamma_{n+1}>> \Hom_S(P^{[n+1]}_{S\,|\,A},\ S) @>\rho^*>> D^{[n+1]}_{S\,|\,A},
\endCD
\]
where the canonical surjection $P^{[n+1]}_{S|A}\onto P^{[n]}_{S|A}$ induces the map in the middle column. The left column realizes local cohomology as the direct limit of Koszul cohomology; the maps in the right column are injective, with $D_{S|A}$ as the directed union.
\end{proof}

\begin{example}
\label{example:identity:operator}
The isomorphism $\rho^*\circ\gamma$ maps the local cohomology element
\[
\eta_S\colonequals\left[\frac{1}{(y_0-x_0) \cdots (y_d-x_d)}\right]
\]
in $H^{d+1}_{\Delta_{S|A}}(P_{S|A})$ to the differential operator in $D_{S|A}$ that is the identity map. More generally, for integers $a_i\ge 0$, the image of the local cohomology element
\[
\left[\frac{1}{(y_0-x_0)^{a_0+1}\cdots(y_d-x_d)^{a_d+1}}\right]\ \in\ H^{d+1}_{\Delta_{S\,|\,A}}(P_{S\,|\,A}),
\]
under $\rho^*\circ\gamma$, is the differential operator $\partial_{a_0,\dots,a_d}\in D_{S|A}$.
\end{example}

There are analogous isomorphisms for differential operators from~$S$ to~$M$, as below:

\begin{proposition}
\label{proposition:iso:S:modules}
For $M$ an $S$-module and $n\ge0$, there are $P_{S|A}$-module isomorphisms
\[
\CD
H^{d+1}(\Delta^{[n+1]}_{S\,|\,A};\ P_{S\,|\,A}\otimes_S M) @>\gamma_n>> \Hom_S(P^{[n]}_{S\,|\,A},\ M) @>\rho^*>> D^{[n]}_{S\,|\,A}(M)
\endCD
\]
and
\[
\CD
H^{d+1}_{\Delta_{S\,|\,A}}(P_{S\,|\,A}\otimes_S M)@>\ \ \rho^*\,\circ\,\gamma\ \ >> D_{S\,|\,A}(M).
\endCD
\]
The maps $\gamma_n$ and $\gamma$ are obtained from those in Propositions~\ref{proposition:iso:koszul} and~\ref{proposition:iso:local:S} by applying $-\otimes_SM$.
\end{proposition}

\begin{proof}
The right exactness of $-\otimes_S M$ gives
\[
H^{d+1}(\Delta^{[n+1]}_{S\,|\,A};\ P_{S\,|\,A}\otimes_S M)\ \cong\ H^{d+1}(\Delta^{[n+1]}_{S\,|\,A};\ P_{S\,|\,A})\otimes_S M\ \cong\ P^{[n]}_{S\,|\,A} \otimes_S M.
\]
Since $P^{[n]}_{S|A}$ is a free $S$-module, one also has the isomorphism
\[
\Hom_S(P^{[n]}_{S\,|\,A},\ M)\ \cong\ \Hom_S(P^{[n]}_{S\,|\,A},\ S)\otimes_S M.
\]
Thus, $\gamma_n$ is an isomorphism, obtained from the corresponding isomorphism in Proposition~\ref{proposition:iso:koszul}; $\rho^*$ is an isomorphism as discussed in \S\ref{subsection:differential operators}. The transition from Koszul cohomology to local cohomology follows as in the proof of Proposition~\ref{proposition:iso:local:S}.
\end{proof}

\subsection{Hypersurfaces}
\label{subsection:hypersurfaces}

As in~\S\ref{subsection:polynomial}, let $S\colonequals A[\bsx]$ be a polynomial ring over a commutative ring~$A$, and identify 
\[
P_{S\,|\,A}\ =\ A[\bsx,\bsy]\ =\ S[\bsy].
\]
Let $R=S/(f(\bsx))$, where $f\in S$ is a nonzero polynomial, and identify $R\otimes_A R$ with
\[
P_{R\,|\,A}\ =\ \frac{A[\bsx,\bsy]}{(f(\bsx),\ f(\bsy))}\ =\ \frac{R[\bsy]}{(f(\bsy))}.
\]
Then $y_0-x_0,\ y_1-x_1,\ \dots,\ y_d-x_d$ serve as generators for $\Delta_{S|A}$, and their images in $P_{R|A}$ are generators for the ideal $\Delta_{R|A}$.

\begin{proposition}
One has $P_{S|A}$-module isomorphisms
\[
D^{[n]}_{R\,|\,A}\ \cong\ \ann\big(f(\bsy),\ D^{[n]}_{S\,|\,A}(R)\big)
\]
for each $n\ge0$, and
\[
D_{R\,|\,A}\ \cong\ \ann\big(f(\bsy),\ D_{S\,|\,A}(R)\big).
\]
\end{proposition}

\begin{proof}
Taking the exact sequence
\[
\CD
P^{[n]}_{S\,|\,A}\otimes_S R @>f(\bsy)>> P^{[n]}_{S\,|\,A}\otimes_S R @>>> P^{[n]}_{R\,|\,A} @>>> 0,
\endCD
\]
and applying $\Hom_R(-,\,R)$, one obtains
\[
\CD
0 @>>> \Hom_R(P^{[n]}_{R\,|\,A},\ R) @>>> \Hom_R(P^{[n]}_{S\,|\,A}\otimes_S R,\ R) @>f(\bsy)>> \Hom_R(P^{[n]}_{S\,|\,A}\otimes_S R,\ R).
\endCD
\]
Since
\[
\Hom_R(P^{[n]}_{R\,|\,A},\ R)\ \cong\ D^{[n]}_{R\,|\,A}
\quad\text{and}\quad
\Hom_R(P^{[n]}_{S\,|\,A}\otimes_S R,\ R)\ \cong\ \Hom_S(P^{[n]}_{S\,|\,A},\ R)\ \cong\ D^{[n]}_{S\,|\,A}(R),
\]
the first isomorphisms follow. The second follows by taking the union over~$n$.
\end{proof}

In light of the previous proposition, we identify $D^{[n]}_{R|A}$ with $\ann\big(f(\bsy),\ D^{[n]}_{S|A}(R)\big)$; this identifies an $A$-linear differential operator $S\to R$ that is annihilated by $f(\bsy)$ with the induced factorization $R\to R$. Next, consider the sequence
\[
\CD
0 @>>> P_{S\,|\,A}\otimes_S R @>f(\bsy)>> P_{S\,|\,A}\otimes_S R @>>> P_{R\,|\,A} @>>> 0,
\endCD
\]
and the induced Koszul cohomology exact sequence
\[
\minCDarrowwidth18pt
\CD
0 @>>> H^d(\Delta^{[n+1]}_{R\,|\,A};\ P_{R\,|\,A}) @>\delta_{f(\bsy)}>> H^{d+1}(\Delta^{[n+1]}_{S\,|\,A};\ P_{S\,|\,A}\otimes_S R)
@>f(\bsy)>> H^{d+1}(\Delta^{[n+1]}_{S\,|\,A};\ P_{S\,|\,A}\otimes_S R),
\endCD
\]
where the injectivity on the left holds since the generators of the ideal $\Delta^{[n+1]}_{S|A}$ form a regular sequence on $P_{S|A}\otimes_S R$. We also use $\delta_{f(\bsy)}$ for the connecting map in the corresponding local cohomology exact sequence, i.e.,
\[
\CD
0 @>>> H^d_{\Delta_{R\,|\,A}}(P_{R\,|\,A}) @>\delta_{f(\bsy)}>> H^{d+1}_{\Delta_{S\,|\,A}}(P_{S\,|\,A}\otimes_S R) @>f(\bsy)>> H^{d+1}_{\Delta_{S\,|\,A}}(P_{S\,|\,A}\otimes_S R).
\endCD
\]

\begin{proposition}
\label{proposition:isomorphism:R}
The maps
\[
\minCDarrowwidth20pt
\CD
H^d(\Delta^{[n+1]}_{R\,|\,A};\ P_{R\,|\,A}) @>\delta_{f(\bsy)}>> H^{d+1}(\Delta^{[n+1]}_{S\,|\,A};\ P_{S\,|\,A}\otimes_S R) @>\gamma_n>> \Hom_S(P^{[n]}_{S\,|\,A},\ R) @>\rho^*>> D^{[n]}_{S\,|\,A}(R)
\endCD
\]
induce an isomorphism of $P_{R|A}$-modules
\[
\CD
H^d(\Delta^{[n+1]}_{R\,|\,A};\ P_{R\,|\,A}) @>\ \ \rho^*\,\circ\,\gamma_n\,\circ\,\delta_{f(\bsy)}\ \ >> D^{[n]}_{R\,|\,A},
\endCD
\]
with $D^{[n]}_{R|A}$ regarded as a submodule of $D^{[n]}_{S|A}(R)$. Passing to the direct limit, the map
\[
\CD
H^d_{\Delta_{R\,|\,A}}(P_{R\,|\,A}) @>\ \ \rho^*\,\circ\,\gamma\,\circ\,\delta_{f(\bsy)}\ \ >> D_{R\,|\,A},
\endCD
\]
is an isomorphism, with $D_{R|A}$ regarded as a submodule of $D_{S|A}(R)$.
\end{proposition}

\begin{proof}
By Proposition~\ref{proposition:iso:S:modules}, the maps $\gamma_n$ and $\rho^*$ are $P_{S|A}$-module isomorphisms. Also,
\[
\delta_{f(\bsy)}\colon H^d(\Delta^{[n+1]}_{R\,|\,A};\ P_{R\,|\,A})\ \to\ \ann\big(f(\bsy),\ H^{d+1}(\Delta^{[n+1]}_{S\,|\,A};\ P_{S\,|\,A}\otimes_S R)\big)
\]
is an isomorphism, and $D^{[n]}_{R|A}=\ann\big(f(\bsy),\ D^{[n]}_{S|A}(R)\big)$. This justifies the first isomorphism. The second is the familiar transition from Koszul cohomology to local cohomology.
\end{proof}

For $S\colonequals A[\bsx]$ a polynomial ring, Example~\ref{example:identity:operator} identifies the element of~$H^{d+1}_{\Delta_{S|A}}(P_{S|A})$ that corresponds to the identity map in $D_{S|A}$; for a hypersurface $R\colonequals A[\bsx]/(f(\bsx))$, we next identify the element of $H^d_{\Delta_{R|A}}(P_{R|A})$ that corresponds to the identity map in $D_{R|A}$:

\begin{example}
\label{example:identity:R}
Since $\mu(f(\bsy))=\mu(f(\bsx))$, where $\mu\colon P_{S|A}\to S$ is the $A$-algebra homomorphism determined by $x_i\mapsto x_i$, and $y_i\mapsto x_i$, one has
\[
f(\bsy)-f(\bsx)\ =\ \sum_{i=0}^d (y_i-x_i)g_i,
\]
where $g_i\in P_{S|A}$. It follows that $\sum_{i=0}^d (y_i-x_i)g_i=0$ in $P_{R|A}$. Thus, we obtain an element
\[
\eta_R\colonequals\left[\cdots,\ \frac{(-1)^ig_i}{\prod_{j\neq i}(y_j-x_j)},\cdots\ \right]\ \in\ H^d_{\Delta_{R\,|\,A}}(P_{R\,|\,A}),
\]
where the signs adhere to the convention in~\S\ref{subsection:koszul:local}. We claim that $\eta_R$ maps to the identity in~$D_{R|A}$ under the isomorphisms in Proposition~\ref{proposition:isomorphism:R}. First note that
\[
\delta_{f(\bsy)}(\eta_R)\ =\ \left[\frac{1}{(y_0-x_0) \cdots (y_d-x_d)}\right].
\]
Using Example~\ref{example:identity:operator}, $\rho^*\circ\gamma$ maps $\delta_{f(\bsy)}(\eta_R)$ to the element of $\Hom_A(S,R)$ corresponding to the canonical surjection $S\onto R$. Thus, $\rho^*\circ\gamma\circ\delta_{f(\bsy)}(\eta_R)$ is the identity map on $R$.
\end{example}

\subsection{Equivalence as \texorpdfstring{$D$}{D}-modules}
\label{subsection:equivalence}

Let $A$ be a ring. Given $A$-algebras $T$ and $T'$, one has a homomorphism $D_{T|A}\to D_{(T\otimes_A T')|A}$ given by
\[
\CD
D_{T\,|\,A} @>\mathrm{id}\otimes1>> D_{T\,|\,A} \otimes_A T' @>>> D_{(T\otimes_A T')\,|\,T'}\ \subseteq\ D_{(T\otimes_A T')\,|\,A}.
\endCD
\]
For an ideal $I$ of $T\otimes_A T'$, we regard the local cohomology module $H^k_I(T\otimes_A T')$ as a $D_{T|A}$-module by restriction of scalars along the map above.

\begin{theorem}
\label{theorem:d:isomorphism:polynomial}
Let $A$ be a commutative ring, and let $S\colonequals A[\bsx]$ be the polynomial ring over~$A$ in the indeterminates $\bsx\colonequals x_0,\dots,x_d$. Then the map
\[
\CD
H^{d+1}_{\Delta_{S\,|\,A}}(P_{S\,|\,A}) @>{\ \ \rho^*\,\circ\,\gamma\ \ }>> D_{S\,|\,A},
\endCD
\]
as in Proposition~\ref{proposition:iso:local:S}, is an isomorphism of $D_{S|A}$-modules.
\end{theorem}

\begin{proof}
By Proposition~\ref{proposition:iso:local:S}, the displayed isomorphism is $P_{S|A}$-linear; we need only verify that it is $D_{S|A}$-linear. In view of Example~\ref{example:identity:operator}, it suffices to verify that the map
\[
\left[\frac{1}{(y_0-x_0)^{a_0+1}\cdots(y_d-x_d)^{a_d+1}}\right]\ \mapsto\ \partial_{a_0,\dots,a_d}
\]
is $D_{S|A}$-linear. In the case $S=A[x]$, i.e., where $d=0$, the verification takes the form
\[
\partial_b\left[\frac{1}{(y-x)^{a+1}}\right]\ =\ \left[\frac{\binom{a+b}{b}}{(y-x)^{a+b+1}}\right]
\ \mapsto\ \binom{a+b}{b}\partial_{a+b}\ =\ \partial_b\partial_a,
\]
with the general case being similar.
\end{proof}

\begin{theorem}
\label{theorem:cyclic:proved}
Let $A$ be a commutative ring, and let $S\colonequals A[\bsx]$ be the polynomial ring over~$A$ where $\bsx\colonequals x_0,\dots,x_d$. Let $R=S/fS$, for $f\in S$ a nonzero polynomial. Then the map
\[
\CD
H^d_{\Delta_{R\,|\,A}}(P_{R\,|\,A}) @>\ \ \rho^*\,\circ\,\gamma\,\circ\,\delta_{f(\bsy)}\ \ >> D_{R\,|\,A},
\endCD
\]
as in Proposition~\ref{proposition:isomorphism:R}, is an isomorphism of $D_{R|A}$-modules.
\end{theorem}

\begin{proof}
We consider $H^{d+1}_{\Delta_{S|A}}(P_{S|A}\otimes_S R)\cong H^{d+1}_{\Delta_{S|A}}(R\otimes_A S)$ as a $D_{R|A}$-module as described at the beginning of this subsection. 
	
Let $\rho^*\circ\gamma\colon H^{d+1}_{\Delta_{S|A}}(P_{S|A}) \to D_{S|A}$ be the $D_{S|A}$-module isomorphism from Theorem~\ref{theorem:d:isomorphism:polynomial}. Applying $D_{S|A}(R)\otimes_{D_{S|A}}(-)$ to this map, we obtain a $D_{R|A}$-linear isomorphism
\[
D_{S\,|\,A}(R)\otimes_{D_{S\,|\,A}} H^{d+1}_{\Delta_{S\,|\,A}}(P_{S\,|\,A}) \to D_{S\,|\,A}(R) \otimes_{D_{S\,|\,A}} D_{S\,|\,A}.
\]
We claim that this map identifies with the isomorphism
\[
H^{d+1}_{\Delta_{S\,|\,A}}(P_{S\,|\,A}\otimes_S R) \to D_{S\,|\,A}(R)
\]
obtained by applying $-\otimes_S R$, as in Proposition~\ref{proposition:iso:S:modules}. Indeed, it is easy to see that
\[
D_{S\,|\,A}(R) \otimes_{D_{S\,|\,A}} H^{d+1}_{\Delta_{S\,|\,A}}(P_{S\,|\,A}) \cong
D_{(S\otimes_A S)\,|\,(A\otimes_A S)}(R\otimes_A S) \otimes_{D_{(S\otimes_A S)\,|\,(A\otimes_A S)}} H^{d+1}_{\Delta_{S\,|\,A}}(P_{S\,|\,A}),
\]
and the latter identifies with $H^{d+1}_{\Delta_{S|A}}(P_{S|A}\otimes_S R)$ by Lemma~\ref{lemma:bimodule}.

We have the commutative diagram 
\[
\CD
0 @>>> H^d_{\Delta_{R\,|\,A}}(P_{R\,|\,A}) @>\delta_{f(\bsy)}>> H^{d+1}_{\Delta_{S\,|\,A}}(P_{S\,|\,A}\otimes_S R) @>{f(\bsy)}>> H^{d+1}_{\Delta_{S\,|\,A}}(P_{S\,|\,A}\otimes_S R)\\
 @. @VV{\cong}V @VV{\cong}V @VV{\cong}V\\
0 @>>> D_{R\,|\,A} @>>> D_{S\,|\,A}(R) @>{f(\bsy)}>> D_{S\,|\,A}(R),
\endCD
\]
where all of maps in the rightmost square are $D_{R|A}$-linear; the first vertical isomorphism is the map from Proposition~\ref{proposition:isomorphism:R}. It follows that this map is a $D_{R|A}$-linear isomorphism. 
\end{proof}

\begin{example}
\label{example:euler}
Let $S\colonequals A[x_0,\dots,x_d]$ be a polynomial ring over a commutative ring~$A$. Fix the $\NN$-grading on $S$ where $S_0=A$ and $\deg x_i=1$ for each $i$. Let $R\colonequals S/(f(\bsx))$, for $f\in S$ a nonzero homogeneous polynomial. We explicitly describe the element of $H^d_{\Delta_{R|A}}(P_{R|A})$ that corresponds to the Euler operator
\[
E\colonequals\sum_{i=0}^d x_i\frac{\partial}{\partial x_i}\ \in\ D_{R\,|\,A}.
\]
With the notation as in Example~\ref{example:identity:R}, let $f(\bsy)-f(\bsx)=\sum_{i=0}^d(y_i-x_i)g_i$ with $g_i\in P_{S|A}$, in which case~$\sum_{i=0}^d(y_i-x_i)g_i=0$ in $P_{R|A}$. Under the isomorphism of the previous theorem, the cohomology class $\eta_R\in H^d_{\Delta_{R|A}}(P_{R|A})$ of the \v Cech cocycle
\[
\left(
\frac{g_0}{\prod_{j\neq 0}(y_j-x_j)},\ \frac{-g_1}{\prod_{j\neq 1}(y_j-x_j)},\ \cdots,\ \frac{(-1)^dg_d}{\prod_{j\neq d}(y_j-x_j)}
\right)
\]
corresponds to the identity element in~$D_{R|A}$. As the isomorphism is one of $D_{R|A}$ modules, the element of $H^d_{\Delta_{R|A}}(P_{R|A})$ that corresponds to the Euler operator $E$ is the cohomology class of the element obtained by applying $E$, considered as an operator on the $\bsx$ variables, componentwise to the above cocycle.
\end{example}

\begin{example}
Consider $R\colonequals A[x]/(x^n)$, for $n$ a positive integer. Then
\[
H^0_{\Delta_{R\,|\,A}}(P_{R\,|\,A})\ =\ P_{R\,|\,A}\ =\ A[x,y]/(x^n,y^n).
\]
In $P_{S|A}$ one has $y^n-x^n=(y-x)\sum_{k=0}^{n-1} y^{n-1-k}x^k$, so the element of~$H^0_{\Delta_{R|A}}(P_{R|A})$ that corresponds to the identity in $D_{R|A}$ is
\[
\sum_{k=0}^{n-1}y^{n-1-k}x^k,
\]
while the element corresponding to the Euler operator is
\[
E\Big(\sum_{k=0}^{n-1}y^{n-1-k}x^k\Big)\ =\ \sum_{k=1}^{n-1}ky^{n-1-k}x^k.
\]
\end{example}

\begin{example}
Set $R\colonequals A[x_0,x_1]/(x_0^2+x_1^2)$. Then the identity in $D_{R|A}$ corresponds to
\[
\left[\frac{y_0+x_0}{y_1-x_1},\ \frac{-(y_1+x_1)}{y_0-x_0}\right]\ \in\ H^1_{\Delta_{R\,|\,A}}(P_{R\,|\,A}),
\]
and the Euler operator to
\[
\left[\frac{y_1x_0+y_0x_1}{(y_1-x_1)^2},\ \frac{-(y_1x_0+y_0x_1)}{(y_0-x_0)^2}\right]\ \in\ H^1_{\Delta_{R\,|\,A}}(P_{R\,|\,A}).
\]
\end{example}

\subsection{Frobenius trace}
\label{subsection:frobenius:trace}

Suppose $A$ is a field of characteristic $p>0$. Let $S\colonequals A[x_0,\dots,x_d]$ be a polynomial ring, and $R\colonequals S/(f(\bsx))$ a graded hypersurface. Let $e$ be a positive integer. The rings $R$ and $R^{p^e}$ are Gorenstein, so the graded analogue of local duality implies that~$\Hom_{R^{p^e}}(R,R^{p^e})$ is a cyclic~$R$-module. To specify a generator, first set
\[
\Phi^e_S\colonequals\partial_{p^e-1,\dots,p^e-1}.
\]
It is a key point that $\Phi^e_S\colon S\to S^{p^e}$, and that it is $S^{p^e}$-linear. Next, consider the composition
\[
\CD
S @>f^{p^e-1}>> S @>\Phi^e_S>> S^{p^e} @>\pi>> R^{p^e}
\endCD
\]
where $\pi$ is the canonical surjection. Since $fS$ is contained in its kernel, the composition factors through a map
\[
\Phi^e_R\colon R\to R^{p^e},
\]
that we define to be the \emph{$e$-th Frobenius trace} of $R$. When $e=1$, we refer to $\Phi_R\colon R\to R^p$ as the Frobenius trace map. In general, the $e$-th Frobenius trace~$\Phi^e_R$ is an element of $D^{[p^e]}_{R|A}$, and is an $R$-module generator for $\Hom_{R^{p^e}}(R,R^{p^e})$.

We claim that $\Phi^e_R$, viewed as a differential operator in $D_{R|A}$, corresponds to the image of the element $\eta_R$ from Example~\ref{example:identity:R} under the $e$-th iterate of the Frobenius action $F$ on the local cohomology module $H^d_{\Delta_{R|A}}(P_{R|A})$, i.e.,
\[
(\rho^*\circ\gamma\circ\delta_{f(\bsy)})(F^e(\eta_R))\ =\ \Phi^e_R.
\]
To see this, consider
\[
\eta_S\colonequals\left[\frac{1}{(y_0-x_0)\cdots(y_d-x_d)}\right]\ \in\ H^{d+1}_{\Delta_{S\,|\,A}}(P_{S\,|\,A})
\]
as in Example~\ref{example:identity:operator}, and note that
\[
F^e(\eta_S)\ =\ \left[\frac{1}{(y_0-x_0)^{p^e}\cdots(y_d-x_d)^{p^e}}\right].
\]
By Example~\ref{example:identity:operator},
\[
(\rho^*\circ\gamma)(F^e(\eta_S))\ =\ \partial_{p^e-1,\dots,p^e-1},
\]
which equals $\Phi^e_S$. Likewise, if $\bar{\eta}_S$ is the image of $\eta_S$ in $H^{d+1}_{\Delta_{S\,|\,A}}(P_{S\,|\,A}\otimes_S R)$, then
\[
(\rho^*\circ\gamma)(F^e(\bar{\eta}_S))\ =\ \pi\circ\Phi^e_S.
\]
With $\eta_R$ as in Example~\ref{example:identity:R}, one has
\[
F^e(\eta_R)\ =\ \left[\cdots,\ \frac{(-1)^i g_i^{p^e}}{\prod_{j\neq i}(y_j-x_j)^{p^e}},\ \cdots\right] \in H^d_{\Delta_{R\,|\,A}}(P_{R\,|\,A}),
\]
so
\[
\delta_{f(\bsy)}(F^e(\eta_R))\ =\ \left[\frac{f(\bsy)^{p^e-1}}{\prod_{i=0}^d (y_i-x_i)^{p^e}}\right]\ =\ f(\bsy)^{p^e-1} F^e(\bar{\eta_S})
\]
in $H^{d+1}_{\Delta_{S|A}}(P_{S|A}\otimes_S R)$. Thus, since $\rho^{*}\circ\gamma$ is $P_{S|A}$-linear, one has
\[
(\rho^{*}\circ\gamma\circ\delta_{f(\bsy)})(F^e(\eta_R))\ =\ f(\bsy)^{p^e-1} \pi\circ\Phi^e_S,
\]
which coincides with $\Phi^e_R$ in $D_{R|A}$, regarded as a submodule of $D_{S|A}(R)$.

\subsection{Lifting differential operators modulo \texorpdfstring{$p$}{p}}
\label{subsection:liftability}

Let $A$, $S$, and $R$ be as in \S\ref{subsection:hypersurfaces}, i.e., $A$ is a commutative ring, $S\colonequals A[x_0,\dots,x_d]$ is a polynomial ring, and $R\colonequals S/(f(\bsx))$ is a hypersurface. Suppose $p>0$ is a prime integer that is regular on $P_{R|A}$. The exact sequence
\[
\CD
0 @>>> P_{R\,|\,A} @>p>> P_{R\,|\,A} @>>> P_{(R/pR)\,|\,(A/pA)} @>>> 0
\endCD
\]
induces a cohomology exact sequence
\[
\minCDarrowwidth20pt
\CD
@>>> H^d_{\Delta_{R\,|\,A}}(P_{{R\,|\,A}}) @>>> H^d_{\Delta_{R\,|\,A}}(P_{(R/pR)\,|\,(A/pA)}) @>\delta_p>> H^{d+1}_{\Delta_{R\,|\,A}}(P_{R\,|\,A}) @>p>> H^{d+1}_{\Delta_{R\,|\,A}}( P_{R\,|\,A})
\endCD
\]
with $\delta_p$ denoting the connecting homomorphism. In particular, $\delta_p$ induces an isomorphism
\[
\CD
\coker\Big(H^d_{\Delta_{R\,|\,A}}(P_{{R\,|\,A}}) @>>> H^d_{\Delta_{R\,|\,A}}(P_{(R/pR)\,|\,(A/pA)})\Big) @>>> \ann\big(p,\ H^{d+1}_{\Delta_{R\,|\,A}}(P_{R\,|\,A})\big).
\endCD
\]
Acknowledging the abuse of notation, we denote the inverse of this isomorphism by $\delta_p^{-1}$, so as to obtain:

\begin{proposition}
\label{proposition:lift:to:Z}
There is an isomorphism of $P_{R|A}$-modules
\[
\CD
\ann\big(p,\ H^{d+1}_{\Delta_{R\,|\,A}}(P_{R\,|\,A})\big) @>>> \coker\Big(D_{R\,|\,A} @>>> D_{(R/pR)\,|\,(A/pA)}\Big),
\endCD
\]
given by $\rho^*\circ\gamma\circ\delta_{f(\bsy)}\circ\delta_p^{-1}$. In particular, given an element 
\[
\nu\ \in\ H^d_{\Delta_{R\,|\,A}}( P_{(R/pR)\,|\,(A/pA)}),
\]
the corresponding differential operator $(\rho^*\circ\gamma\circ\delta_{f(\bsy)})(\nu)\in D_{(R/pR)|(A/pA)}$ lifts to an element of $D_{R|A}$ if and only if $\delta_p(\nu)= 0$.
\end{proposition}

\begin{proof}
One has a commutative diagram
\[
\CD
H^d_{\Delta_{R\,|\,A}}(P_{R\,|\,A}) @>>> H^d_{\Delta_{R\,|\,A}}(P_{(R/pR)\,|\,(A/pA)})\\
@VVV @VVV\\
D_{R\,|\,A} @>>> D_{(R/pR)\,|\,(A/pA)}
\endCD
\]
where the vertical maps are the isomorphisms $\rho^*\circ\gamma\circ\delta_{f(\bsy)}$ of Proposition~\ref{proposition:isomorphism:R}. Combining with the isomorphism
\[
\CD
\ann\big(p,\ H^{d+1}_{\Delta_{R\,|\,A}}(P_{R\,|\,A})\big) @>\delta_p^{-1}>> \coker\Big(H^d_{\Delta_{R\,|\,A}}(P_{R\,|\,A}) @>>> H^d_{\Delta_{R\,|\,A}}(P_{(R/pR)\,|\,(A/pA)})\Big),
\endCD
\]
the assertion follows.
\end{proof}

Next, we establish a similar criterion for when a differential operator $D_{(R/pR)|(A/pA)}$ lifts to $D_{(R/p^2R)|(A/p^2A)}$. Start with the exact sequence
\[
\CD
0 @>>> P_{(R/pR)\,|\,(A/pA)} @>p>> P_{(R/p^2R)\,|\,(A/p^2A)} @>>> P_{(R/pR)\,|\,(A/pA)} @>>> 0
\endCD
\]
and the corresponding cohomology exact sequence
\begin{multline*}
\minCDarrowwidth20pt
\CD
@>>> H^d_{\Delta_{R\,|\,A}}(P_{(R/p^2R)\,|\,(A/p^2A)}) @>>> H^d_{\Delta_{R\,|\,A}}(P_{(R/pR)\,|\,(A/pA)})
\endCD \\
\minCDarrowwidth20pt
\CD
@>\beta_p>> H^{d+1}_{\Delta_{R\,|\,A}}(P_{(R/pR)\,|\,(A/pA)}) @>p>> H^{d+1}_{\Delta_{R\,|\,A}}(P_{(R/p^2R)\,|\,(A/p^2A)}) @>>>.
\endCD
\end{multline*}
The connecting homomorphism $\beta_p$ in the sequence above is the Bockstein homomorphism on local cohomology, see~\S\ref{subsection:bockstein}. The map $\beta_p$ induces an isomorphism
\[
\CD
\coker\Big(H^d_{\Delta_{R\,|\,A}}(P_{(R/p^2R)\,|\,(A/p^2A)}) @>>> H^d_{\Delta_{R\,|\,A}}(P_{(R/pR)\,|\,(A/pA)})\Big) @>>> \image(\beta_p),
\endCD
\]
and we denote the inverse of this isomorphism by $\beta_p^{-1}$. Then, along the same lines as Proposition~\ref{proposition:lift:to:Z}, we have:

\begin{proposition}
\label{proposition:lift:to:Zp2}
There is an isomorphism of $P_{R|A}$-modules
\[
\CD
\image(\beta_p) @>>> \coker\Big(D_{(R/p^2R)\,|\,(A/p^2A)} @>>> D_{(R/pR)\,|\,(A/pA)}\Big)
\endCD
\]
induced by $\rho^*\circ\gamma\circ\delta_{f(\bsy)}\circ\beta_p^{-1}$. In particular, for an element 
\[
\nu\in H^d_{\Delta_{R\,|\,A}}(P_{(R/pR)\,|\,(A/pA)}),
\]
the corresponding differential operator $(\rho^*\circ\gamma\circ\delta_{f(\bsy)})(\nu)\in D_{(R/pR)|(A/pA)}$ lifts to an element of $D_{(R/p^2R)|(A/p^2A)}$ if and only if $\beta_p(\nu)= 0$.
\end{proposition}

\subsection{Lifting Frobenius trace}
\label{subsection:lift:frobenius:trace}

Let $S\colonequals\ZZ[\bsx]$ be a polynomial ring over $\ZZ$ in the indeterminates $\bsx\colonequals x_0,\dots,x_d$. Identifying $P_S$ with $\ZZ[\bsx,\bsy]$ as before, note that $P_S$ may also be viewed as a polynomial ring over $\ZZ$ in the indeterminates
\[
x_0,\ \dots,\ x_d,\ y_0-x_0,\ \dots,\ y_d-x_d.
\]
Fix a prime integer $p>0$, and let $\Lambda_p\colon P_S \to P_S$ denote the standard lift of Frobenius with respect to the indeterminates above, i.e., $\Lambda_p$ is the endomorphism of $P_S$ with $\Lambda_p(x_i)=x_i^p$ and~$\Lambda_p(y_i-x_i)=(y_i-x_i)^p$ for each $i$. Let $\phi_p$ denote the corresponding $p$-derivation on $P_S$.

Let $R=S/(f(\bsx))$ be a graded hypersurface. Assume that $p$ is regular on~$P_R$, in which case there is an exact sequence
\begin{equation}
\label{equation:trace:cd}
\minCDarrowwidth20pt
\CD
H^d_{\Delta_R}(P_R) @>>> H^d_{\Delta_R}(P_R/pP_R) @>\delta_p>> H^{d+1}_{\Delta_R}(P_R) @>p>> H^{d+1}_{\Delta_R}(P_R) @>{\pi_p}>> H^{d+1}_{\Delta_R}(P_R/pP_R).
\endCD
\end{equation}
Let $\eta_R\in H^d_{\Delta_R}(P_R)$ be the element that corresponds to the identity map in $D_R$, as in Example~\ref{example:identity:R}, and let $\bar{\eta}_R$ denote its image in $H^d_{\Delta_R}(P_R/pP_R)$. We claim that 
\[
\delta_p\big(F(\bar{\eta}_R)\big)\ =\ \left[\frac{\phi_p \big(f(\bsy)-f(\bsx)\big)}{\prod_{i=0}^d (y_i-x_i)^p}\right]
\]
as elements of $H^{d+1}_{\Delta_R}(P_R)$. To see this, take $g_i\in P_S$ with $f(\bsy)-f(\bsx)=\sum_{i=0}^d (y_i-x_i)g_i$~as in Example~\ref{example:identity:R}. Then, in $P_S$, one has
\begin{align*}
\phi_p\big(f(\bsy)-f(\bsx)\big)\ &=\ \frac{1}{p}\Big(\Lambda_p\big(f(\bsy)-f(\bsx)\big)\ -\ \big(f(\bsy)-f(\bsx)\big)^p\Big)\\
&=\ \frac{1}{p}\Big(\sum_i\Lambda_p(y_i-x_i)\Lambda_p(g_i)\ -\ \big(f(\bsy)-f(\bsx)\big)^p\Big)\\
&=\ \frac{1}{p}\Big(\sum_i(y_i-x_i)^p\Lambda_p(g_i)\ -\ \big(f(\bsy)-f(\bsx)\big)^p\Big)\\
&=\ \frac{1}{p}\Big(\sum_i(y_i-x_i)^p\big(g_i^p+p\phi_p(g_i)\big)\ -\ \big(f(\bsy)-f(\bsx)\big)^p\Big)\\
&=\ \sum_i(y_i-x_i)^p\phi_p(g_i)\ +\ \frac{1}{p}\Big(\sum_i(y_i-x_i)^pg_i^p - \big(f(\bsy)-f(\bsx)\big)^p\Big).
\end{align*}
It follows that the image of $\phi_p\big(f(\bsy)-f(\bsx)\big)$ in $P_R$, i.e., modulo the ideal $(f(\bsx),f(\bsy))$, is
\[
\sum_i(y_i-x_i)^p\phi_p(g_i)\ +\ \frac{1}{p}\sum_i(y_i-x_i)^p g_i^p.
\]
Hence, in $H^{d+1}_{\Delta_R}(P_R)$, one has
\begin{align*}
\left[\frac{\phi_p\big(f(\bsy)-f(\bsx)\big)}{\prod_i (y_i-x_i)^p}\right]\ 
&=\ \left[\frac{\sum_i(y_i-x_i)^p\phi_p (g_i)+\frac{1}{p}\sum_i(y_i-x_i)^pg_i^p}{\prod_i(y_i-x_i)^p}\right]\\
&=\ \left[\frac{\frac{1}{p}\sum_i(y_i-x_i)^pg_i^p}{\prod_i(y_i-x_i)^p}\right]\\
&=\ \delta_p\big(F(\bar{\eta}_R)\big),
\end{align*}
which proves the claim. Similarly, with $\beta_p$ denoting the Bockstein homomorphism
\[
\CD
H^d_{\Delta_R}(P_R/pP_R) @>{\ \ \pi_p\, \circ \, \delta_p\ \ }>> H^{d+1}_{\Delta_R}(P_R/pP_R)
\endCD
\]
with $\delta_p$ and $\pi_p$ as in~\eqref{equation:trace:cd}, one see that
\[
\beta_p\big(F(\bar{\eta}_R)\big)\ =\ \left[\frac{\phi_p\big(f(\bsy)-f(\bsx)\big)}{\prod_{i=0}^d (y_i-x_i)^p}\right]
\]
as elements of $H^{d+1}_{\Delta_R}(P_R/pP_R)$.

Combining the preceding calculations with Propositions~\ref{proposition:lift:to:Z}~and~\ref{proposition:lift:to:Zp2}, we obtain:

\begin{theorem}
\label{theorem:lifting:trace}
Let $S\colonequals\ZZ[\bsx]$ be a polynomial ring in the indeterminates $\bsx\colonequals x_0,\dots,x_d$. Fix a nonzero homogeneous polynomial $f(\bsx)$ in $S$, and set $R\colonequals S/(f(\bsx))$. Let $p$ be a prime integer that is a regular element on $P_R$. Let $\Lambda_p$ be the endomorphism of $P_S\colonequals\ZZ[\bsx,\bsy]$ with
\[
\Lambda_p(x_i)=x_i^p\quad\text{ and }\quad\Lambda_p(y_i-x_i)=(y_i-x_i)^p
\]
for each $i$, and let $\phi_p$ denote the corresponding $p$-derivation.

Then the Frobenius trace map~$\Phi_{R/pR}$ on~$R/pR$, as in \S\ref{subsection:frobenius:trace}, lifts to a differential operator on $R$ if and only if the local cohomology element
\[
\left[\frac{\phi_p\big(f(\bsy)-f(\bsx)\big)}{\prod_{i=0}^d (y_i-x_i)^p}\right]\ \in\ H^{d+1}_{\Delta_R}(P_R)
\]
is zero; similarly, $\Phi_{R/pR}$ lifts to a differential operator on $R/p^2R$ if and only if the image of the displayed element in $H^{d+1}_{\Delta_R}(P_R/pP_R)$ is zero.
\end{theorem}

\begin{remark}
For $S$ and $R$ as in the previous theorem set $a\colonequals\deg f(\bsx)-(d+1)$. Then, with the grading shifts as below, Theorem~\ref{theorem:cyclic:proved} provides degree-preserving isomorphisms
\[
D_R\ \cong\ H^d_{\Delta_R}(P_R)(a)\quad\text{ and }\quad D_{R/pR}\ \cong\ H^d_{\Delta_R}(P_R/pP_R)(a).
\]
In particular, the identity element in the ring $D_R$ or in the ring $D_{R/pR}$ corresponds to a cohomology class in ${[H^d_{\Delta_R}(P_R)]}_a$ or ${[H^d_{\Delta_R}(P_R/pP_R)]}_a$ respectively; confer Example~\ref{example:identity:R}.

It is a straightforward calculation that the Frobenius trace map $\Phi_{R/pR}$ on $R/pR$ is a differential operator of degree $ap-a$. It follows that $\Phi_{R/pR}$ corresponds to a cohomology class in ${[H^d_{\Delta_R}(P_R/pP_R)]}_{ap}$ as, indeed, is evident by the calculation in~\S\ref{subsection:frobenius:trace}.
\end{remark}

\begin{remark}
Consider $R\colonequals S/fS$, where $S\colonequals\ZZ[x_0,x_1,x_2]$ and $f(\bsx)\colonequals x_0^3+x_1^3+x_2^3$. Let $p\neq3$ be a prime integer. By \cite[\S3]{Smith} and~\cite[Example~6.6]{Jeffries}, the Frobenius trace map~$\Phi_{R/pR}$ on~$R/pR$ does not lift to a differential operator on $R$. Using Theorem~\ref{theorem:lifting:trace}, it follows that the local cohomology element
\[
\zeta_p\colonequals\left[\frac{\phi_p\big(f(\bsy)-f(\bsx)\big)}{\prod(y_i-x_i)^p}\right]
\]
in $H^3_{\Delta_R}(P_R)$ is nonzero. Note that $\zeta_p$ is annihilated by $p$ in view of~\eqref{equation:trace:cd}. Hence, the module $H^3_{\Delta_R}(P_R)$ contains a nonzero $p$-torsion element for each prime integer $p\neq 3$.

Likewise, consider $R\colonequals S/fS$, where $S\colonequals\ZZ[x_0,x_1,x_2,x_3]$ and $f(\bsx)\colonequals x_0^3+x_1^3+x_2^3+x_3^3$. Recent work of Mallory \cite[Theorem~1.2]{Mal} shows that $R$ has no differential operators of negative degree. For each prime integer $p\neq 3$, the Frobenius trace on $R/pR$ has negative degree, namely $1-p$, and hence does not lift to $R$. It follows that $H^4_{\Delta_R}(P_R)$ contains a nonzero $p$-torsion element for each prime integer $p\neq 3$.
\end{remark}

\section{Integer torsion in local cohomology}
\label{section:torsion}

Our approach to constructing mod $p$ differential operators that do not lift, based on the results of the previous section, is via constructing $p$-torsion elements in local cohomology modules of the form $H^{d+1}_{\Delta_R}(P_R)$. Towards this, the following lemma will prove useful; the notation is as in~\S\ref{subsection:lift:frobenius:trace}.

\begin{lemma}
\label{lemma:mainlemma}
Let $S\colonequals\ZZ[\bsx]$ be a polynomial ring in the indeterminates $\bsx\colonequals x_0,\dots,x_d$. Consider a polynomial $f(\bsx)\in S$ of the form 
\[
f(\bsx)\ =\ \sum_{i=0}^mx_if_i,
\]
where $m\le d$ and each $f_i\in S$. Set $R\colonequals S/(f(\bsx))$. Fix a prime integer $p$, and let $\phi_p$ denote the $p$-derivation of $P_S$ as in Theorem~\ref{theorem:lifting:trace}; this restricts to the standard $p$-derivation of $S$ with respect to $\bsx$.

If the local cohomology element $\displaystyle{\left[\frac{\phi_p\big(f(\bsx)\big)}{(x_0\cdots x_m)^p}\right]}$ in $H^{m+1}_{(x_0,\dots,x_m)}(R)$ is nonzero, then so is the element $\displaystyle{\left[\frac{\phi_p\big(f(\bsy)-f(\bsx)\big)}{\prod_{i=0}^d (y_i-x_i)^p}\right]}$ in $H^{d+1}_{\Delta_R}(P_R)$.

Similarly, if the image of $\displaystyle\left[\frac{\phi_p\big(f(\bsx)\big)}{(x_0\cdots x_m)^p}\right]$ in $H^{m+1}_{(x_0,\dots,x_m)}(R/pR)$ is nonzero, then so is the image of $\displaystyle \left[\frac{\phi_p\big(f(\bsy)-f(\bsx)\big)}{\prod_i(y_i-x_i)^p}\right]$ in $H^{d+1}_{\Delta_R}(P_R/pP_R)$.
\end{lemma}

\begin{proof}
Recall the notation $\bsy\colonequals y_0,\dots,y_d$. We identify $P_S$ with the polynomial ring $\ZZ[\bsx,\bsy]$ and $P_R$ with~$\ZZ[\bsx,\bsy]/(f(\bsx), f(\bsy))$, so that $\Delta_R$ is the ideal generated by the elements
\[
y_0-x_0,\ y_1-x_1,\ \dots,\ y_d-x_d.
\]
Suppose the displayed element of $H^{d+1}_{\Delta_R}(P_R)$ is zero. Then there exists an integer $k\ge 0$ such that in the ring $P_S$ one has
\begin{multline}
\label{equation:lemma:1}
\frac{\Lambda_p\big(f(\bsy)-f(\bsx)\big)-\big(f(\bsy)-f(\bsx)\big)^p}{p}(y_0-x_0)^k\cdots(y_d-x_d)^k\\
\in\ \big(f(\bsy),\ f(\bsx),\ (y_0-x_0)^{p+k},\ \dots,\ (y_d-x_d)^{p+k}\big)P_S.
\end{multline}
Let $P'$ denote the image of $P_S$ under the specialization $x_i\mapsto 0$ for $0\le i\le m$. Note that~$f(\bsx)\mapsto 0$ under this specialization. Since one has a commutative diagram
\[
\CD
P_S@>\Lambda_p>>P_S\\
@VVV @VVV\\
P'@>\Lambda_p>>P',
\endCD
\]
the ideal membership~\eqref{equation:lemma:1} specializes to give
\begin{multline}
\label{equation:lemma:2}
\frac{\Lambda_p\big(f(\bsy)\big)-\big(f(\bsy)\big)^p}{p}\ y_0^k\cdots y_m^k\ (y_{m+1}-x_{m+1})^k\cdots(y_d-x_d)^k\\
\in\ \big(f(\bsy),\ y_0^{p+k},\ \dots,\ y_m^{p+k},\ (y_{m+1}-x_{m+1})^{p+k},\ \dots,\ (y_d-x_d)^{p+k}\big)P'.
\end{multline}
The elements $y_{m+1}-x_{m+1},\ \dots,\ y_d-x_d$ are algebraically independent over 
\[
\ZZ[\bsy]/\big(f(\bsy),\ y_0^{p+k},\ \dots,\ y_m^{p+k}\big),
\]
so $y_{m+1}-x_{m+1},\ \dots,\ y_d-x_d$, as well as their powers, form a regular sequence in the ring
\[
P'/\big(f(\bsy),\ y_0^{p+k},\ \dots,\ y_m^{p+k}\big)P'.
\]
Using this,~\eqref{equation:lemma:2} implies
\begin{multline*}
\frac{\Lambda_p\big(f(\bsy)\big)-\big(f(\bsy)\big)^p}{p}\ y_0^k\cdots y_m^k\\
\in\ \big(f(\bsy),\ y_0^{p+k},\ \dots,\ y_m^{p+k},\ (y_{m+1}-x_{m+1})^p,\ \dots,\ (y_d-x_d)^p\big)P'.
\end{multline*}
Next, specialize $x_i\mapsto y_i$ for $m+1\le i\le d$, to obtain
\[
\frac{{\Lambda}_p\big(f(\bsy)\big)-\big(f(\bsy)\big)^p}{p}\ y_0^k\cdots y_m^k\ \in\ \big(f(\bsy),\ y_0^{p+k},\ \dots,\ y_m^{p+k}\big)\ZZ[\bsy],
\]
where $\Lambda_p$ is the standard lift of Frobenius on $\ZZ[\bsy]$ with respect to $\bsy$. Renaming $y_i\mapsto x_i$ for each $i$, the above reads
\[
\phi_p\big(f(\bsx)\big)x_0^k\cdots x_m^k\ \in\ \big(f(\bsx),\ x_0^{p+k},\ \dots,\ x_m^{p+k}\big)\ZZ[\bsx],
\]
implying that
\[
\left[\frac{\phi_p\big(f(\bsx)\big)}{(x_0\cdots x_m)^p}\right]\ \in\ H^{m+1}_{(x_0,\dots,x_m)}(R)
\]
is zero. The proof of the final assertion is similar.
\end{proof}

A question of Huneke~\cite[Problem~4]{Huneke:Sundance} asks whether local cohomology modules of Noetherian rings have finitely many associated prime ideals. This was answered in the negative by the second author in \cite[\S4]{Singh:MRL}: there is hypersurface $R$ over $\ZZ$ for which a local cohomology module~$H^k_\fraka(R)$ has $p$-torsion elements for each prime integer $p$. From this, it follows that~$H^k_\fraka(R)$ has infinitely many associated primes; this is extended to several natural families of hypersurfaces by Theorems~\ref{theorem:quadratic},~\ref{theorem:pfaffian}, and~\ref{theorem:determinant}. On the other hand, for a polynomial ring over $\ZZ$, or, more generally, a smooth $\ZZ$-algebra, the answer to Huneke's question is positive; this follows from the following result, which we use in the next section:

\begin{theorem}\cite[Theorem~3.1]{BBLSZ}
\label{theorem:bblsz}
Let $S$ be a smooth $\ZZ$-algebra, and $\fraka$ an ideal of $S$ generated by elements $\bsf=f_1,\dots,f_t$. Let $k$ be a nonnegative integer.

If a prime integer is a nonzerodivisor on the Koszul cohomology module $H^k(\bsf;\,S)$, then it is also a nonzerodivisor on the local cohomology module $H^k_\fraka(S)$.
\end{theorem}

\section{Toric rings}

We consider the question whether differential operators lift modulo~$p$ for invariant rings of tori. An $m\times n$ integer matrix $M$ determines a map
\[
\CD
\ZZ^m @<M<< \ZZ^n,
\endCD
\]
and hence a normal affine semigroup
\[
\Lambda\colonequals\NN^n \cap \ker M.
\]
Let $A$ be a commutative ring. The semigroup ring $A[\Lambda]$ is the subring of the polynomial ring $A[\NN^n]\colonequals A[x_1,\dots,x_n]$ that is generated, as an $A$-algebra, by monomials with exponent vectors in $\Lambda$. The ring $A[\NN^n]$ has an $\NN^m$-grading where the degree of $x_i$ is the $i$-th column of~$M$. Under this grading, the degree zero component of $A[\NN^n]$ is precisely $A[\Lambda]$, which is hence a direct summand of $A[\NN^n]$ as an $A[\Lambda]$-module. We refer to a ring of the form $A[\Lambda]$, or an isomorphic copy, as a \emph{toric $A$-algebra}.

When $A$ has an infinite group of units $A^{\times}$, the ring $A[\Lambda]$ is the invariant ring for an action of a product of copies of $A^{\times}$ on $A[\NN^n]$, hence the name.

\begin{example}
\label{example:toric}
The integer matrix $M\colonequals\begin{pmatrix} 1 & 1 & -2 \end{pmatrix}$ yields the semigroup $\Lambda$ generated by
\[
\begin{pmatrix}2\\ 0\\ 1\end{pmatrix},\quad
\begin{pmatrix}1\\ 1\\ 1\end{pmatrix},\quad
\begin{pmatrix}0\\ 2\\ 1\end{pmatrix},
\]
so $\ZZ[\Lambda]$ is the subring $\ZZ[x_1^2x_3,\ x_1x_2x_3,\ x_2^2x_3]$ of the polynomial ring $\ZZ[x_1,x_2,x_3]$.
\end{example}

\begin{theorem}
\label{theorem:toric}
Let $R$ be a toric $\ZZ$-algebra. Then, for each prime $p$, each differential operator on $R/pR$ lifts to a differential operator on $R$; thus, the natural map
\[
D_R\otimes_\ZZ(\ZZ/p\ZZ)\ \to\ D_{R/pR}
\]
is an isomorphism.
\end{theorem}

\begin{proof}
Write $R=\ZZ[\Lambda]$, for $\Lambda$ as in the earlier discussion. Recall that~$R$ is a direct summand of $S\colonequals\ZZ[\NN^n]$, and let $\rho\colon S\to R$ be a splitting. Then $\rho$ induces a splitting~$\bar{\rho}$ of the inclusion $R/pR\subseteq S/pS$, with the commutative diagram
\[
\CD
S @>\rho>> R\\
@VVV @VVV\\
S/pS @>\bar{\rho}>> R/pR.
\endCD
\]
Fix $\delta\in D_{R/pR}$. By \cite[Proposition~6.6]{BJNB}, there exists $\hat{\delta}\in D_{S/pS}$ such that $\hat{\delta}|_{R/pR} = \delta$. Since $S$ is a polynomial ring, $\hat{\delta}$ lifts to a differential operator $\xi\in D_S$, see for example \cite[Lemma~2.1]{BBLSZ}. By~\cite[Proposition~3.1]{Smith}, the composition $\rho\circ(\xi|_R)$ is an element of $D_R$. It is readily verified that this maps to $\delta$ under $D_R\to D_{R/pR}$.
\end{proof}

For a toric $\ZZ$-algebra $R$, while the map $D_R\to D_{R/pR}$ is indeed surjective, a differential operator on $R/pR$ need not lift to a differential operator on $R$ of the same order:

\begin{example}
The ring $\ZZ[\Lambda]$ in Example~\ref{example:toric} is readily seen to be isomorphic to 
\[
R\colonequals\ZZ[x_1^2,\ x_1x_2,\ x_2^2],
\]
which is the second Veronese subring of the polynomial ring $T\colonequals\ZZ[x_1,x_2]$.

The group $\{\pm 1\}$ acts on $T\otimes_\ZZ\CC=\CC[x_1,x_2]$, with the invariant ring being $R\otimes_\ZZ\CC$, and the ring of differential operators $D_{R\otimes_\ZZ\CC|\CC}$ is generated by the elements of~$D_{T\otimes_\ZZ\CC|\CC}$ that have even degree with respect to the standard grading on $T\otimes_\ZZ\CC$. Note that the derivations
\[
x_i\frac{\partial}{\partial x_j}\ \in\ D_{R\otimes_\ZZ\CC|\CC}
\]
have degree $0$; there are no derivations of negative degree, confer~\cite{Kantor2}. Since
\[
D^1_R\ \subseteq\ D^1_{R\otimes_\ZZ\CC|\CC},
\]
it follows that $R$ has no derivations of negative degree. In contrast, the ring $R/2R$ has a derivation of negative degree, namely
\[
x_1^{-1}\frac{\partial}{\partial x_2}\ =\ x_2^{-1}\frac{\partial}{\partial x_1},
\]
which is the endomorphism of $R/2R$ with
\begin{align*}
x_1^{2i}x_2^{2j} & \mapsto 0\\
x_1^{2i+1}x_2^{2j+1} & \mapsto x_1^{2i}x_2^{2j}
\end{align*}
for $i,j\in\NN$. This derivation lifts to the differential operator $\frac{\partial}{\partial x_1}\frac{\partial}{\partial x_2}$ in $D^2_R$.
\end{example}

\section{Quadratic forms and Pfaffians}
\label{section:pfaffian}

The main focus of this section is Pfaffian hypersurfaces, but we begin with the analogue of Theorem~\ref{theorem:main} for quadratic hypersurfaces, where the arguments are most transparent:

\begin{theorem}
\label{theorem:quadratic}
Let $x_1,\dots,x_{2m}$ be indeterminates over $\ZZ$, where $m\ge 3$, and set
\[
R\colonequals\ZZ[x_1,\dots,x_{2m}]/(\sum_{i=1}^m x_ix_{m+i}).
\]
Then, for each prime integer $p>0$, the Frobenius trace on $R/pR$ does not lift to a differential operator on $R/p^2 R$, nor, a fortiori, to a differential operator on $R$.
\end{theorem}

\begin{proof}
Set $S\colonequals\ZZ[x_1,\dots,x_{2m}]$. Fix $p$, and let $\phi_p$ be the standard $p$-derivation on $S$ with respect to $\bsx$. In view of Theorem~\ref{theorem:lifting:trace} and Lemma~\ref{lemma:mainlemma}, it suffices to prove that the element
\[
\left[\frac{\phi_p(\sum_{i=1}^m x_ix_{m+i})}{(x_1\cdots x_m)^p}\right]\ \in\ H^m_{(x_1,\dots,x_m)}(R/pR)
\]
is nonzero. Indeed, if it were zero, then there exists an integer $k\ge0$ such that
\[
\phi_p(\sum_{i=1}^m x_ix_{m+i})(x_1\cdots x_m)^k\ \in\ \big(x_1^{p+k},\ \dots,\ x_m^{p+k}\big)R/pR.
\]
The image of $\phi_p(\sum x_ix_{m+i})$ in $R$ equals $\frac{1}{p}\sum x_i^p x_{m+i}^p$, regarded as an element of $R$, giving
\begin{equation}
\label{equation:quadratic:1}
\Big(\frac{1}{p}\sum_{i=1}^m x_i^p x_{m+i}^p\Big)(x_1\cdots x_m)^k\ \in\ \big(x_1^{p+k},\ \dots,\ x_m^{p+k}\big)R/pR.
\end{equation}
Let $e_i\in\ZZ^m$ denote the $i$-th standard basis vector. Consider the $\ZZ^m$-grading of $R/pR$ where
\begin{align*}
\deg x_i\ &=\ e_i,\\
\deg x_{m+i}\ &=\ -e_i,
\end{align*}
for $1\le i\le m$. The element on the left hand side in~\eqref{equation:quadratic:1} then has degree $(k,\dots,k)$. Hence, in a homogeneous equation for the ideal membership~\eqref{equation:quadratic:1}, the coefficient of $x_i^{p+k}$ on the right has degree $(k,\dots,k,-p,k,\dots,k)$, and therefore the coefficient must be a multiple of 
\[
x_1^k\cdots x_{i-1}^k x_{i+1}^k\cdots x_m^k x_{m+i}^p,
\]
i.e.,
\[
\Big(\frac{1}{p}\sum_{i=1}^m x_i^p x_{m+i}^p\Big)(x_1\cdots x_m)^k\ \in\ 
\big(x_2^k\cdots x_m^k x_{m+1}^px_1^{p+k},\ \dots,\ x_1^k\cdots x_{m-1}^k x_{2m}^p x_m^{p+k}\big)R/pR.
\]
Canceling the term $(x_1\cdots x_m)^k$, the above display implies that
\begin{equation}
\label{equation:quadratic:2}
\frac{1}{p}\sum_{i=1}^m x_i^p x_{m+i}^p\ \in\ \big((x_1x_{m+1})^p,\ \dots,\ (x_mx_{2m})^p\big)
\end{equation}
in the ring
\[
{(R/pR)}_{(0,\dots,0)}\ =\ \ZZ/p\ZZ[x_1x_{m+1},\dots,x_mx_{2m}]/(\sum_{i=1}^m x_ix_{m+i}).
\]
But ${(R/pR)}_{(0,\dots,0)}$ may be identified with the polynomial ring $\ZZ/p\ZZ[z_1,\dots,z_{m-1}]$, where
\[
z_i\colonequals x_ix_{m+i}\qquad\text{for }\ 1\le i\le m-1,
\]
in which case,~\eqref{equation:quadratic:2} reads
\[
\frac{z_1^p+\dots+z_{m-1}^p+(-1)^p(z_1+\dots+z_{m-1})^p}{p}\ \in\ (z_1^p,\ \dots,\ z_{m-1}^p).
\]
This is readily seen to be false, for example by examining the coefficient of $z_1^{p-1}z_2$.
\end{proof}

We now turn to Pfaffian hypersurfaces. Let $n$ be an even integer with~$n\ge 4$, and let~$Z$ be an $(n-2)\times n$ matrix of indeterminates over an infinite field $K$. Set~$T\colonequals K[Z]$. The symplectic group~$\Sp_{n-2}(K)$ acts $K$-linearly on $T$ by the rule
\[
M\colon Z\mapsto MZ\qquad\text{for }\ M\in\Sp_{n-2}(K).
\]
By \cite[\S6]{DCP}, the invariant ring for this action is the $K$-algebra generated by the entries of the product matrix $Z^{\tr}\,\Omega\,Z$, where $\Omega$ is the size $n-2$ standard symplectic block matrix
\[
\begin{pmatrix} 0 & I \\ -I & 0 \end{pmatrix},
\]
with $I$ the identity. This invariant ring is isomorphic to $K[X]/(\pf X)$ for $X$ an~$n \times n$ alternating matrix of indeterminates, and $\pf X$ its Pfaffian. When $K$ has characteristic zero, the ring of differential operators on the invariant ring is described explicitly in \cite[IV~1.9,~Case~C]{LS}.

For $M$ an alternating matrix, the cofactor expansion for Pfaffians takes the form
\[
\pf M\ =\ \sum_{j\ge 2} (-1)^j m_{1j}\pf M_{1j},
\]
where $M_{1j}$ is the submatrix obtained by deleting the first and $j$-th rows and columns. For~$t$ an even integer, we use $\Pf_t(M)$ to denote the ideal generated by the Pfaffians of the~$t\times t$ principal submatrices of $M$. It follows from the cofactor expansion that $\Pf_t(M)\subseteq\Pf_{t-2}(M)$.

The Pfaffian of a $4\times 4$ alternating matrix of indeterminates $X$ is the quadratic form
\[
x_{12}x_{34} - x_{13}x_{24} + x_{14}x_{23},
\]
which, aside from the change of notation, coincides with the case $m=3$ in Theorem~\ref{theorem:quadratic}. More generally one has the following, which completes the proof of Theorem~\ref{theorem:main}~\eqref{main:b} using Theorem~\ref{theorem:lifting:trace} and Lemma~\ref{lemma:mainlemma}.

\begin{theorem}
\label{theorem:pfaffian}
Let $X$ be an $n\times n$ alternating matrix of indeterminates over $\ZZ$, where $n$ is an even integer with $n\ge 4$. Set~$S\colonequals\ZZ[X]$ and~$R\colonequals S/(\pf X)$. Fix a prime integer $p>0$, and let $\phi_p$ be the standard $p$-derivation on $S$ with respect to the indeterminates $X$. Then
\[
\left[\frac{\phi_p(\pf X)}{(x_{12}\cdots x_{1n})^p}\right]\ \in\ H^{n-1}_{(x_{12},\dots,x_{1n})}(R)
\]
is a nonzero $p$-torsion element; moreover, its image in $H^{n-1}_{(x_{12},\dots,x_{1n})}(R/pR)$ is nonzero.

In particular, the local cohomology module $H^{n-1}_{(x_{12},\dots,x_{1n})}(R)$ contains a nonzero $p$-torsion element for each prime integer $p>0$.
\end{theorem}

\begin{proof}
As noted above, the case $n=4$ is settled by the proof of Theorem~\ref{theorem:quadratic}. We proceed by induction on $n$, with $n=4$ serving as the base case.

The image of the displayed element is zero in $H^{n-1}_{(x_{12},\dots,x_{1n})}(R/pR)$ if and only if there exists an integer $k\ge0$, such that in the polynomial ring $S$ one has
\begin{equation}
\label{equation:pfaffian}
\phi_p(\pf X) (x_{12}\cdots x_{1n})^k\ \in\ \big(p,\ \pf X,\ x_{12}^{p+k},\ \dots,\ x_{1n}^{p+k}\big)S.
\end{equation}
We know this cannot happen for $n=4$. Suppose~\eqref{equation:pfaffian} holds for an even integer~$n\ge 6$. Specialize $x_{n-1,n} \mapsto 1$, and $x_{ij}\mapsto 0$ for $i=2,\dots,n-2$ and $j=n-1,n$, in which case, the image of $X$ is
\[
\begin{pmatrix}
0 & x_{12} & x_{13} & x_{14} & \hdots & x_{1,n-3} & x_{1,n-2} & x_{1,n-1} & x_{1n}\\
-x_{12} & 0 & x_{23} & x_{24} & \hdots & x_{2,n-3} & x_{2,n-2} & 0 & 0\\
-x_{13} & -x_{23} & 0 & x_{34} & \hdots & x_{3,n-3} & x_{3,n-2} & 0 & 0\\
-x_{14} & -x_{24} & x_{34} & 0 & \hdots & x_{4,n-3} & x_{4,n-2} & 0 & 0\\
\vdots & \vdots & \vdots & \vdots & & \vdots & \vdots & \vdots & \vdots \\
-x_{1,n-3} & -x_{2,n-3} & -x_{3,n-3} & -x_{4,n-3} & \hdots & 0 & x_{n-3,n-2} & 0 & 0\\
-x_{1,n-2} & -x_{2,n-2} & -x_{3,n-2} & -x_{4,n-2} & \hdots & -x_{n-3,n-2}& 0 & 0 & 0\\
-x_{1,n-1} & 0 & 0 & 0 & \hdots & 0 & 0 & 0 & 1\\
-x_{1n} & 0 & 0 & 0 & \hdots & 0 & 0 & -1 & 0\\
\end{pmatrix}.
\]
The upper left $(n-2)\times(n-2)$ submatrix is unchanged; denote this by $X'$. Elementary row and column operations transform the matrix displayed above to
\[
\begin{pmatrix} X' & 0\\ 0 & \Omega \end{pmatrix},
\qquad\text{ where }\
\Omega\colonequals \begin{pmatrix} 0 & 1\\ -1 & 0 \end{pmatrix},
\]
which shows that the Pfaffian of the image of $X$ after specialization equals $\pf X'$. Hence the ideal membership~\eqref{equation:pfaffian} specializes to
\[
\phi_p(\pf X') (x_{12}\cdots x_{1n})^k\ \in\ \big(p,\ \pf X',\ x_{12}^{p+k},\ \dots,\ x_{1n}^{p+k}\big)S',
\]
with $S'$ denoting the image of $S$ under the specialization. The indeterminates $x_{1,n-1}$ and $x_{1n}$ do not occur in the polynomial $\pf X'$, and hence form a regular sequence on
\[
S'/\big(p,\ \pf X',\ x_{12}^{p+k},\ \dots,\ x_{1,n-2}^{p+k}\big)S'.
\]
Using this, we obtain
\[
\phi_p(\pf X') (x_{12}\cdots x_{1,n-2})^k\ \in\ \big(p,\ \pf X',\ x_{12}^{p+k},\ \dots,\ x_{1,n-2}^{p+k},\ x_{1,n-1}^p,\ x_{1n}^p\big)S'.
\]
Next specialize $x_{1,n-1}\mapsto 0$ and $x_{1n}\mapsto 0$, so as to obtain 
\[
\phi_p(\pf X') (x_{12}\cdots x_{1,n-2})^k\ \in\ \big(p,\ \pf X',\ x_{12}^{p+k},\ \dots,\ x_{1,n-2}^{p+k}\big)S'',
\]
where $S''$ is the image of $S'$ under the specialization, equivalently, the polynomial ring over~$\ZZ$ in the indeterminates occurring in $X'$. But this violates the inductive hypothesis.
\end{proof}

We next show that Pfaffian hypersurfaces do not admit a lift of Frobenius modulo $p^2$:

\begin{theorem}
\label{theorem:pfaffian:frobenius}
Let $X$ be an $n\times n$ alternating matrix of indeterminates over $\ZZ$, where $n$ is an even integer with $n\ge 4$. Set~$S\colonequals\ZZ[X]$ and~$R\colonequals S/(\pf X)$. Fix a prime integer $p>0$. Then the Frobenius endomorphism on $R/pR$ does not lift to an endomorphism of $R/p^2R$.
\end{theorem}

\begin{proof}
We proceed by induction on even integers $n\ge 4$, along the same lines as in the proof of Theorem~\ref{theorem:pfaffian}. The case $n=4$ is the hypersurface $R$ defined by 
\[
x_{12}x_{34} - x_{13}x_{24} + x_{14}x_{23}.
\]
While the case of diagonal quadratic forms of rank at least $5$, in odd characteristic, is covered by \cite[Theorem~4.15]{Zdanowicz}, we give a different argument to avoid characteristic restrictions. Suppose that the Frobenius endomorphism on $R/pR$ lifts to $R/p^2R$. Then, by Proposition~\ref{proposition:zdanowicz},
one has
\begin{equation}
\label{equation:quadratic:frobenius}
\frac{1}{p}\big(x_{12}^px_{34}^p - x_{13}^px_{24}^p + x_{14}^px_{23}^p\big)\ \in\ 
\big(x_{12}^p,\ x_{13}^p,\ x_{14}^p,\ x_{23}^p,\ x_{24}^p,\ x_{34}^p\big)R/pR.
\end{equation}
Using the grading
\begin{alignat*}3
\deg x_{12}\ &=\ e_1, \qquad\qquad & \deg x_{34}\ &=\ -e_1,\\
\deg x_{13}\ &=\ e_2, \qquad\qquad & \deg x_{24}\ &=\ -e_2,\\
\deg x_{14}\ &=\ e_3, \qquad\qquad & \deg x_{23}\ &=\ -e_3,
\end{alignat*}
as in the proof of Theorem~\ref{theorem:quadratic}, we may work in the subring
\[
{(R/pR)}_{(0,0,0)}\ =\ \ZZ/p\ZZ[x_{12}x_{34}, \ x_{13}x_{24},\ x_{14}x_{23} ]/(x_{12}x_{34}-x_{13}x_{24}+x_{14}x_{23}),
\]
which, in turn, may be identified with the polynomial ring $\ZZ/p\ZZ[z_1,\ z_2]$, where
\[
z_1\colonequals x_{12}x_{34}\quad \text{ and }\quad z_2\colonequals x_{13}x_{24}. 
\]
But then~\eqref{equation:quadratic:frobenius} reads
\[
\frac{1}{p}\big(z_1^p - z_2^p + (z_2-z_1)^p\big)\ \in\ \big(z_1^p,\ z_2^p\big)R/pR,
\]
which is seen to be false by examining the coefficient of $z_1^{p-1}z_2$.

For the inductive step, let $R\colonequals S/(\pf X)$ be the hypersurface defined by the Pfaffian of an $n\times n$ alternating matrix $X$, for an even integer $n\ge 6$. Suppose that the Frobenius endomorphism on $R/pR$ lifts to $R/p^2R$. Then, by Proposition~\ref{proposition:zdanowicz},
\[
\phi_p(\pf X)\ \in\ \bigg(\Big(\frac{\partial\pf X}{\partial x_{ij}}\Big)^p : 1\le i<j\le n \bigg)\ S/(p,\,\pf X)S.
\]
Each partial derivative $\partial\pf X/\partial x_{ij}$ is, up to sign, the Pfaffian of the $(n-2)\times (n-2)$ submatrix of $X$ obtained by deleting rows $i,j$ and columns $i,j$. Using the notation
\[
\fraka^{[p]}\colonequals (a^p\mid a\in\fraka)
\]
for ideals $\fraka$ in a ring of prime characteristic $p>0$, the preceding ideal membership may hence be written as 
\[
\phi_p(\pf X)\ \in\ \Pf_{n-2}(X)^{[p]}\ S/(p,\,\pf X)S.
\]
Applying the specialization in the proof of Theorem~\ref{theorem:pfaffian}, $X$ may be replaced by 
\[
X''\colonequals \begin{pmatrix} X' & 0\\ 0 & \Omega \end{pmatrix},
\qquad\text{ where }\ 
\Omega\colonequals\begin{pmatrix} 0 & 1\\ -1 & 0 \end{pmatrix},
\]
implying that
\[
\phi_p(\pf X')\ \in\ \Pf_{n-2}(X'')^{[p]}\ S'/(p,\,\pf X')S'.
\]
But
\[
\Pf_{n-2}(X'') \ =\ \Pf_{n-4}(X')\ =\ \bigg(\Big(\frac{\partial\pf X'}{\partial x_{ij}}\Big): 1\le i<j\le n-2 \bigg),
\]
contradicting the inductive hypothesis.
\end{proof}

\section{Determinantal hypersurfaces}
\label{section:generic:determinant}

Let $K$ be an infinite field. Let $Y$ and $Z$ be $n \times (n-1)$ and $(n-1) \times n$ matrices of indeterminates respectively, and set $T\colonequals K[Y,Z]$. The general linear group $\GL_{n-1}(K)$ acts~$K$-linearly on $T$ where, for $M\in\GL_{n-1}(K)$, one has
\[
M\colon\begin{cases} Y & \mapsto YM^{-1}\\ Z & \mapsto MZ.\end{cases}
\]
By \cite[\S3]{DCP}, the invariant ring for this action is generated over $K$ by the entries of the product matrix $YZ$. This invariant ring is isomorphic to $K[X]/(\det X)$, where $X$ is an $n\times n$ matrix of indeterminates. When $K$ has characteristic zero, the ring of differential operators on the invariant ring is described explicitly in \cite[IV~1.9,~Case~A]{LS}.

We complete the proof of Theorem~\ref{theorem:main}~\eqref{main:a}. In view of Theorem~\ref{theorem:lifting:trace} and Lemma~\ref{lemma:mainlemma}, it suffices to prove:

\begin{theorem}
\label{theorem:determinant}
Let $X$ be an $n\times n$ matrix of indeterminates over $\ZZ$, where $n\ge 3$. Fix a prime integer $p>0$, and let $\phi_p$ be the standard $p$-derivation on $S\colonequals\ZZ[X]$ with respect to~$X$. Take~$\fraka$ to be the ideal~$(x_{12},\dots,x_{1n},x_{22},\dots,x_{2n})\,R$, where $R\colonequals S/(\det X)$. Then
\[
\left[\frac{\phi_p(\det X)}{(x_{12}\cdots x_{1n}x_{22}\cdots x_{2n})^p}\right]\ \in\ H^{2n-2}_\fraka(R)
\]
is a nonzero $p$-torsion element; moreover, its image in $H^{2n-2}_\fraka(R/pR)$ is nonzero.

In particular, the local cohomology module $H^{2n-2}_\fraka(R)$ contains a nonzero~$p$-torsion element for each prime integer $p>0$.
\end{theorem}

\begin{proof}
We first show that the proof of the theorem reduces to the case $n=3$. Fix $p$. Suppose that the image of the displayed element in $H^{2n-2}_\fraka(R/pR)$ is zero. Then there exists an integer~$k\ge0$ such that in the polynomial ring $S$ one has
\begin{equation}
\label{equation:generic:determinant:1}
\phi_p(\det X) (x_{12}\cdots x_{1n}x_{22}\cdots x_{2n})^k\ \in\ 
\big(p,\ \det X,\ x_{12}^{p+k},\ \dots,\ x_{1n}^{p+k},\ x_{22}^{p+k},\ \dots,\ x_{2n}^{p+k})S.
\end{equation}
For $i=4,\dots,n$, specialize $x_{ii} \mapsto 1$ and $x_{ij} \mapsto 0$ for $j\neq i$. The image of $X$ after specialization has the form
\[
\begin{pmatrix}
X' & * \\
0 & I
\end{pmatrix},
\]
where $X'$ denotes the upper left $3\times 3$ submatrix of $X$, and~$I$ is the size~$n-3$ identity matrix. Note that $\det X$ specializes to $\det X'$, and~$\phi_p(\det X)$ to $\phi_p(\det X')$. Hence the ideal membership~\eqref{equation:generic:determinant:1} specializes to
\[
\phi_p(\det X') (x_{12}\cdots x_{1n}x_{22}\cdots x_{2n})^k\ \in\ \big(p,\ \det X',\ x_{12}^{p+k},\ \dots,\ x_{1n}^{p+k},\ x_{22}^{p+k},\ \dots,\ x_{2n}^{p+k}\big)S',
\]
with $S'$ the image of $S$. Since the indeterminates $x_{14},\dots,x_{1n},x_{24},\dots,x_{2n}$ do not occur in the polynomials $\det X'$ and $\phi_p(\det X')$, the ideal membership above implies
\[
\phi_p(\det X') (x_{12}x_{13}x_{22}x_{23})^k\ \in\ \big(p,\ \det X',\ x_{12}^{p+k},\ x_{13}^{p+k},\ x_{22}^{p+k},\ x_{23}^{p+k}\big)S'.
\]
Specializing $x_{ij}\mapsto 0$ for $i\ge 4$ and also for $j\ge 4$, we reduce to the case $n=3$; specifically, it suffices to consider
\[
X\colonequals\begin{pmatrix}
x_{11} & x_{12} & x_{13}\\
x_{21} & x_{22} & x_{23}\\
x_{31} & x_{32} & x_{33}
\end{pmatrix}
\]
and $S=\ZZ[X]$, in which case~\eqref{equation:generic:determinant:1} reads
\begin{equation}
\label{equation:generic:determinant:2}
\phi_p(\det X) (x_{12}x_{13}x_{22}x_{23})^k\ \in\ 
\big(p,\ \det X,\ x_{12}^{p+k},\ x_{13}^{p+k},\ x_{22}^{p+k},\, x_{23}^{p+k}\big)S.
\end{equation}
Specialize $x_{31}\mapsto 0$. Using $X''$ for the image of $X$, one has
\[
\det X''\ =\ -(x_{21}x_{33})x_{12}\ +\ (x_{21}x_{32})x_{13}\ +\ (x_{11}x_{33})x_{22}\ -\ (x_{11}x_{32})x_{23}.
\]
In the hypersurface
\[
R'\colonequals\ZZ[x_{11},x_{12},x_{13},x_{21},x_{22},x_{23},x_{32},x_{33}]/(\det X''),
\]
set
\[
\lambda\colonequals\frac{1}{p}\big(-x_{21}^px_{33}^px_{12}^p\ +\ x_{21}^px_{32}^px_{13}^p\ +\ x_{11}^px_{33}^px_{22}^p\ -\ x_{11}^px_{32}^px_{23}^p\big).
\]
Then~\eqref{equation:generic:determinant:2} implies that
\[
\lambda(x_{12}x_{13}x_{22}x_{23})^k\ \in\ \big(x_{12}^{p+k},\ x_{13}^{p+k},\ x_{22}^{p+k},\, x_{23}^{p+k}\big)R'/pR'.
\]
Take the $\ZZ^5$-grading on $R'/pR'$ defined by
\[
\begingroup
\setlength{\tabcolsep}{2pt}
\begin{tabular}{lcrrrrr}
$\deg\ x_{11}$ &= & ($0$, & $0$, & $0$, & $1$, & $0$) \\
$\deg\ x_{12}$ &= & ($1$, & $0$, & $-1$, & $0$, & $-1$) \\
$\deg\ x_{13}$ &= & ($1$, & $-1$, & $0$, & $0$, & $-1$) \\
$\deg\ x_{21}$ &= & ($0$, & $0$, & $0$, & $0$, & $1$) \\
$\deg\ x_{22}$ &= & ($1$, & $0$, & $-1$, & $-1$, & $0$) \\
$\deg\ x_{23}$ &= & ($1$, & $-1$, & $0$, & $-1$, & $0$) \\
$\deg\ x_{32}$ &= & ($0$, & $1$, & $0$, & $0$, & $0$) \\
$\deg\ x_{33}$ &= & ($0$, & $0$, & $1$, & $0$, & $0$) \\
\end{tabular}
\endgroup
\]
and note that $\lambda$ is homogeneous of degree $(p,0,0,0,0)$. Hence
\[
\deg\lambda(x_{12}x_{13}x_{22}x_{23})^k\ =\ (p+4k,\ -2k,\ -2k,\ -2k,\ -2k).
\]
Fix a homogeneous equation
\[
\lambda(x_{12}x_{13}x_{22}x_{23})^k\ =\ \alpha x_{12}^{p+k} + \beta x_{13}^{p+k} + \gamma x_{22}^{p+k} + \delta x_{23}^{p+k}
\]
with $\alpha,\beta,\gamma,\delta$ in $R'/pR'$. The element $\alpha$ has degree
\[
(3k,\ -2k,\ p-k,\ -2k,\ p-k).
\]
Let $\mu$ be a monomial of the above degree. Examining the second and fourth components of the degree, the exponents on $x_{13}$ and $x_{23}$ in $\mu$ add up to at least $2k$, as do the exponents on $x_{22}$ and $x_{23}$. Bear in mind the first component, $\mu$ must be a multiple of $x_{23}^k$. A similar analysis for $\beta$, $\gamma$, and $\delta$ shows that
\[
\lambda(x_{12}x_{13}x_{22}x_{23})^k\ \in\ \big(x_{23}^kx_{12}^{p+k},\ x_{22}^kx_{13}^{p+k},\ x_{13}^kx_{22}^{p+k},\ x_{12}^kx_{23}^{p+k}\big)R'/pR'.
\]
Next, specialize $x_{11}$ and $x_{33}$ to $1$ and $x_{21}$ and $x_{32}$ to $-1$, in which case one has
\begin{equation}
\label{equation:generic:determinant:3}
\lambda'(x_{12}x_{13}x_{22}x_{23})^k\ \in\ \big(x_{23}^kx_{12}^{p+k},\ x_{22}^kx_{13}^{p+k},\ x_{13}^kx_{22}^{p+k},\ x_{12}^kx_{23}^{p+k}\big)B,
\end{equation}
where $\lambda'$ denote the image of $\lambda$ under the specialization, and $B$ is the image of $R'/pR'$, i.e.,
\[
B\colonequals\ZZ/p\ZZ[x_{12},x_{13},x_{22},x_{23}]/(x_{12}+x_{13}+x_{22}+x_{23}),
\]
which we identify with the polynomial ring $\ZZ/p\ZZ[x_{13},x_{22},x_{23}]$. With this identification,
\begin{align*}
\lambda'\ =\ \frac{1}{p}\big(x_{12}^p + x_{13}^p + x_{22}^p + x_{23}^p\big)\
& =\ \frac{1}{p}\big((-x_{13}-x_{22}-x_{23})^p + x_{13}^p + x_{22}^p + x_{23}^p\big)\\
& =\ \pm x_{13}^{p-1}x_{22}+\cdots
\end{align*}
is a polynomial in $\ZZ/p\ZZ[x_{13},x_{22},x_{23}]$ in which each indeterminate occurs with exponents strictly less than $p$. The ring $B$ is a free module over its subring
\[
B^p=\ZZ/p\ZZ[x_{13}^p,\ x_{22}^p,\ x_{23}^p],
\]
with a basis given by monomials in $x_{13},\ x_{22},\ x_{23}$, with each exponent less than $p$. Let
\[
\pi\colon B\to B^p
\]
be the $B^p$-linear map that sends $x_{13}^{p-1}x_{22}$ to $1$ and other basis elements to $0$. Specifically,
\[
\pi(\lambda')=\pm 1.
\]
Consider~\eqref{equation:generic:determinant:3} where, without loss of generality, the exponent $k$ is taken to be a power of~$p$, and apply $\pi$. Then, in the ring $B^p$, one has
\begin{equation}
\label{equation:generic:determinant:4}
(x_{12}x_{13}x_{22}x_{23})^k\ \in\ \big(x_{23}^kx_{12}^{p+k},\ x_{22}^kx_{13}^{p+k},\ x_{13}^kx_{22}^{p+k},\ x_{12}^kx_{23}^{p+k}\big),
\end{equation}
where we retain the notation $x_{12}^p=-(x_{13}^p+x_{22}^p+x_{23}^p)$ for the sake of symmetry; note that~$B^p$ may be regarded as a polynomial ring in any \emph{three} of the elements
\[
x_{12}^p,\ x_{13}^p,\ x_{22}^p,\ x_{23}^p.
\]
The ideal membership~\eqref{equation:generic:determinant:4} implies the existence of elements $a,b,c,d$ in $B^p$ with
\[
(x_{12}x_{13}x_{22}x_{23})^k\ =\ ax_{23}^kx_{12}^{p+k}+ bx_{22}^kx_{13}^{p+k} + cx_{13}^kx_{22}^{p+k} + dx_{12}^kx_{23}^{p+k}.
\]
Rearranging terms, one has
\[
x_{12}^k((x_{13}x_{22}x_{23})^k - ax_{23}^kx_{12}^p - dx_{23}^{p+k})\ \in\ (x_{22}^kx_{13}^{p+k},\ x_{13}^kx_{22}^{p+k}).
\]
But $x_{12}^k$ is a nonzerodivisor in the ring $B^p/(x_{22}^kx_{13}^{p+k},\ x_{13}^kx_{22}^{p+k})$, so the above implies that
\[
(x_{13}x_{22}x_{23})^k\ \in\ (x_{23}^kx_{12}^p,\ x_{22}^kx_{13}^{p+k},\ x_{13}^kx_{22}^{p+k},\ x_{23}^{p+k}).
\]
Similarly, using that $x_{13}^k$ is a nonzerodivisor modulo in $B^p/(x_{23}^kx_{12}^p,\ x_{23}^{p+k})$, one obtains
\[
(x_{22}x_{23})^k\ \in\ (x_{23}^kx_{12}^p,\ x_{22}^kx_{13}^p,\ x_{22}^{p+k},\ x_{23}^{p+k}).
\]
Continuing, $x_{22}^k$ is a nonzerodivisor in $B^p/(x_{23}^kx_{12}^p,\ x_{23}^{p+k})$, yielding
\[
x_{23}^k\ \in\ (x_{23}^kx_{12}^p,\ x_{13}^p,\ x_{22}^p,\ x_{23}^{p+k}),
\]
and finally, with $x_{23}^k$ being a nonzerodivisor in $B^p/(x_{13}^p,\ x_{22}^p)$, one obtains the contradiction
\[
1\ \in\ (x_{12}^p,\ x_{13}^p,\ x_{22}^p,\ x_{23}^p).\qedhere
\]
\end{proof}

Determinantal hypersurfaces, in general, do not admit a lift of Frobenius modulo~$p^2$:

\begin{theorem}
\label{theorem:determinant:frobenius}
Let $X$ be an $n\times n$ matrix of indeterminates over $\ZZ$, where $n\ge 3$. Fix a prime integer $p>0$. Set~$S\colonequals\ZZ[X]$ and~$R\colonequals S/(\det X)$. Then the Frobenius endomorphism on~$R/pR$ does not lift to an endomorphism of $R/p^2R$.
\end{theorem}

\begin{proof}
As in the proof of Theorem~\ref{theorem:determinant}, one first reduces to the case $n=3$ as follows. Suppose that the Frobenius endomorphism on $R/pR$ lifts to $R/p^2R$. Then, by Proposition~\ref{proposition:zdanowicz},
\[
\phi_p(\det X)\ \in\ \bigg(\Big(\frac{\partial\det X}{\partial x_{ij}}\Big)^p : 1\le i,j\le n \bigg)\ S/(p,\,\det X)S.
\]
Each partial derivative above is, up to sign, the determinant of an $(n-1)\times (n-1)$ submatrix of $X$, so the above may be restated as
\begin{equation}
\label{equation:frobenius:determinant1}
\phi_p(\det X)\ \in\ I_{n-1}(X)^{[p]}\ S/(p,\,\det X)S,
\end{equation}
where $I_{n-1}(X)$ denotes the ideal generated by the size $n-1$ minors of $X$. Applying the specialization $S\to S'$ in the proof of Theorem~\ref{theorem:determinant}, one has
\[
X\mapsto \begin{pmatrix}X' & *\\ 0 & I\end{pmatrix},
\]
with $I$ denoting the identity matrix of size $n-3$. Hence $I_{n-1}(X)S'=I_2(X')$, and the ideal membership~\eqref{equation:frobenius:determinant1} specializes to
\[
\phi_p(\det X')\ \in\ I_2(X')^{[p]}\ S'/(p,\,\det X')S',
\]
which is essentially the $n=3$ case.

Assume $n=3$. Specializing $x_{13}$ and $x_{31}$ to $0$, the resulting matrix
\[
X''\colonequals\begin{pmatrix}
x_{11} & x_{12} & 0\\
x_{21} & x_{22} & x_{23}\\
0 & x_{32} & x_{33}
\end{pmatrix}
\]
has determinant $x_{11}x_{22}x_{33}-x_{11}x_{23}x_{32}-x_{12}x_{21}x_{33}$, and the ideal membership implies
\begin{equation}
\label{equation:frobenius:determinant2}
\phi_p(\det X'')\ \in\ I_2(X'')^{[p]}
\end{equation}
in the ring
\[
\ZZ/p\ZZ[x_{11},\ x_{12},\ x_{21},\ x_{22},\ x_{23},\ x_{32},\ x_{33}]/(\det X'').
\]
Using the $\ZZ^5$-grading in the proof of Theorem~\ref{theorem:determinant}, $\det X''$ has degree~$(p,0,0,0,0)$, so we obtain an ideal membership in the subring generated by elements of degree~$(*,0,0,0,0)$, namely the ring
\begin{multline*}
\ZZ/p\ZZ[x_{11}x_{22}x_{33},\ x_{11}x_{23}x_{32},\ x_{12}x_{21}x_{33}]/(x_{11}x_{22}x_{33}-x_{11}x_{23}x_{32}-x_{12}x_{21}x_{33})\\
\cong\ \ZZ/p\ZZ[x_{11}x_{23}x_{32},\ x_{12}x_{21}x_{33}].
\end{multline*}
Working with the degree of each generator of $I_2(X'')$, the statement~\eqref{equation:frobenius:determinant2} gives
\[
\frac{1}{p}\big((x_{11}x_{23}x_{32}+x_{12}x_{21}x_{33})^p-(x_{11}x_{23}x_{32})^p-(x_{12}x_{21}x_{33})^p\big)\ \in\ \big(x_{11}x_{23}x_{32},\ x_{12}x_{21}x_{33}\big)^{[p]}
\]
in the ring above, which is readily seen to be false.
\end{proof}
	
\section{Symmetric determinantal hypersurfaces}
\label{section:symmetric}

Let $Z$ be an $(n-1) \times n$ matrix of indeterminates over an infinite field $K$. Set~$T\colonequals K[Z]$. The orthogonal group $\O_{n-1}(K)$ acts $K$-linearly on $T$ by the rule
\[
M\colon Z\mapsto MZ\qquad\text{for }\ M\in\O_{n-1}(K).
\]
By \cite[\S5]{DCP}, the invariant ring for this action is the $K$-algebra generated by the entries of the product matrix $Z^{\tr}Z$, and is isomorphic to $K[X]/(\det X)$ for $X$ an $n\times n$ symmetric matrix of indeterminates. When $K$ has characteristic zero, the ring of differential operators on this invariant ring is described explicitly in \cite[IV~1.9~Case~B]{LS}.

\begin{theorem}
\label{theorem:symmetric:3:odd}
Let $X$ be a symmetric $3\times 3$ matrix of indeterminates over $\ZZ$, and set $R$ to be the hypersurface $\ZZ[X]/(\det X)$. Then, for each odd prime integer $p$, the map
\[
D_R\ \to\ D_{R/pR}
\]
is surjective.
\end{theorem}

The map displayed above is not surjective when $p=2$, see Theorem~\ref{theorem:symmetric:char:2}. 

\begin{proof}
In view of Proposition~\ref{proposition:lift:to:Z}, it suffices to show that each odd prime $p$ acts injectively on the local cohomology module $H^6_{\Delta_R}(P_R)$. This is unaffected by inverting the integer $2$, so we work instead with the rings $\ZZ_2=\ZZ[1/2]$ and
\[
R_2=\ZZ_2[X]/(\det X).
\]
Let $T\colonequals\ZZ[u,v,w,x,y,z]$ be a polynomial ring, and $T_2=T[1/2]$. It is readily checked that the symmetric matrix
\[
\begin{pmatrix}
2ux & uy+vx & uz+wx\\ uy+vx & 2vy & vz+wy\\ uz+wx & vz+wy & 2wz
\end{pmatrix}
\]
has determinant $0$. By a dimension argument, it follows that $R_2$ is isomorphic to the subring
\[
\ZZ_2[ux,\ vy,\ wz,\ uy+vx,\ uz+wx,\ vz+wy].
\]
of $T_2$ and, indeed, we identify $R_2$ with this subring.

We claim that $R_2$ is a direct summand of~$T_2$ as an $R_2$-module. To see this, first consider the~$\ZZ$-grading on $T_2$ where the indeterminates $u,v,w$ have degree $1$, and $x,y,z$ have degree~$-1$. It follows that the degree $0$ component of $T_2$, i.e., the ring
\[
B\colonequals\ZZ_2[ux,\ uy,\ uz,\ vx,\ vy,\ vz,\ wx,\ wy,\ wz]
\]
is a direct summand of $T_2$ as a $B$-module. Next, let $G\colonequals\langle\sigma\rangle$ be a group of order $2$ acting on $T_2$, where $\sigma$ is the involution with
\[
u\mapsto x,\qquad v\mapsto y,\qquad w\mapsto z.
\]
The action of $G$ on $T_2$ restricts to an action on the subring $B$, and
\[
R_2\ \subseteq\ B^G.
\]
We claim that equality holds in the above display. The equation
\[
(uy)^2-(uy)(uy+vx)+(ux)(vy)\ =\ 0
\]
shows that $uy$ is integral over $R_2$; similarly one sees that $B$ is integral over $R_2$. Moreover, it is readily checked that at the level of fraction fields one has
\[
\ffield(R_2)(uy)\ = \ffield(B),
\]
so that $[\ffield(B):\ffield(R_2)]\le 2$. Hence
\[
\ffield(R_2)=\ffield\big(B^G\big).
\]
Each element of $B^G$ is integral over $R_2$, and belongs to the fraction field of $R_2$; but $R_2$ is normal, so we conclude that $R_2=B^G$.

Since $2$ is a unit in $R_2$, the Reynolds operator $B\to R_2$ shows that $R_2$ is a direct summand of $B$ as an~$R_2$-module. It follows that $R_2$ is a direct summand of $T_2$ as an $R_2$-module. At this stage, one may invoke \cite[Theorem~6.3]{Jeffries} to conclude that the map $D_R\ \to\ D_{R/pR}$ is surjective for \emph{almost all} integer primes~$p$; we shall, however, go further and prove that the map is surjective for each odd prime integer $p$. Fix such a prime $p$; it suffices to prove that $p$ acts injectively on the local cohomology module
\[
H^6_{\Delta}(R_2\otimes_{\ZZ_2}R_2),
\]
where $\Delta$ is the diagonal ideal in $R_2\otimes_{\ZZ_2}R_2$. Since $R_2$ is a direct summand of $T_2$ as an $R_2$-module, it follows that~$R_2\otimes_{\ZZ_2}R_2$ is a direct summand of $T_2\otimes_{\ZZ_2}T_2$ as an~$R_2\otimes_{\ZZ_2}R_2$-module. It suffices to show that $p$ acts injectively on $H^6_{\Delta}(C)$, where $C\colonequals T_2\otimes_{\ZZ_2}T_2$ is identified with
\[
\ZZ_2[u,v,w,x,y,z,\ u',v',w',x',y',z']
\]
and $\Delta$ is the ideal of $C$ generated by the elements
\begin{multline*}
ux-u'x',\quad vy-v'y',\quad wz-w'z',\\
uy+vx-u'y'-v'x',\quad uz+wx-u'z'-w'x',\quad vz+wy-v'z'-w'y'.
\end{multline*}
Consider the ideals of $C$ as below:
\[
\frakp\colonequals I_2\begin{pmatrix}
u & v & w & x' & y' & z'\\
u' & v' & w' & x & y & z\end{pmatrix}
\quad\textrm{and}\quad
\frakq\colonequals I_2\begin{pmatrix}
u & v & w & u' & v' & w'\\
x' & y' & z' & x & y & z\end{pmatrix}.
\]
Since $\frakp$ and $\frakq$ are generated by minors of matrices of indeterminates, each is prime; moreover, $\frakp$ and $\frakq$ each contain $\Delta$. It is a straightforward---albeit slightly tedious---verification that $\frakp^2\frakq\subseteq\Delta$. It then follows that
\[
\frakp\cap\frakq\ =\ \rad\Delta.
\]
Since $H^k_\frakp(C)=0=H^k_\frakq(C)$ for integers $k$ other than $5$ and $9$, the Mayer-Vietoris sequence
\[
\minCDarrowwidth20pt
\CD
@>>> H^6_\frakp(C)\oplus H^6_\frakq(C) @>>> H^6_\Delta(C) @>>> H^7_{\frakp+\frakq}(C) @>>> H^7_\frakp(C)\oplus H^7_\frakq(C) @>>>
\endCD
\]
gives an isomorphism $H^6_\Delta(C)\cong H^7_{\frakp+\frakq}(C)$. It now suffices to check that $p$ acts injectively on $H^7_{\frakp+\frakq}(C)$. We claim that the ideal $(\frakp+\frakq)C/pC$ has height $7$. To see this, note that there is an isomorphism
\[
C/(\frakp+\frakq+pC)\ \to\ \ZZ/p\ZZ[x_1,x_2]\ \#\ \ZZ/p\ZZ[y_1,y_2]\ \#\ \ZZ/p\ZZ[z_1,z_2,z_3],
\]
with $\#$ denoting the Segre product of graded rings, given by
\begin{alignat*}2
u \mapsto x_1y_1z_1, &\qquad v \mapsto x_1y_1z_2, &\qquad w \mapsto x_1y_1z_3, \\
u' \mapsto x_1y_2z_1, &\qquad v' \mapsto x_1y_2z_2, &\qquad w' \mapsto x_1y_2z_3, \\
x \mapsto x_2y_2z_1, &\qquad y \mapsto x_2y_2z_2, &\qquad z \mapsto x_2y_2z_3, \\
x' \mapsto x_2y_1z_1, &\qquad y' \mapsto x_2y_1z_2, &\qquad z' \mapsto x_2y_1z_3.
\end{alignat*}
It follows that $H^6_{\frakp+\frakq}(C/pC)=0$, which gives the desired injectivity using the exactness of
\[
\CD
@>>> H^6_{\frakp+\frakq}(C/pC) @>>> H^7_{\frakp+\frakq}(C) @>p>> H^7_{\frakp+\frakq}(C) @>>> H^7_{\frakp+\frakq}(C/pC).
\endCD\qedhere
\]
\end{proof}

\begin{remark}
\label{remark:orthogonal}
Over a field $K$ of odd characteristic, the orthogonal group $\O_2(K)$ is linearly reductive: to see this, first enlarge $K$ so that it contains $i\colonequals\sqrt{-1}$. The special orthogonal group $\SO_2(K)$ is then isomorphic to a torus $K^\times$, as seen by conjugating elements of 
\[
\SO_2(K)\ =\ \left\{\begin{pmatrix} a & b \\ -b & a \end{pmatrix}\ \Big|\ a^2+b^2=1,\ a,b\in K\right\}
\]
by the matrix $\begin{pmatrix} 1 & i \\ i & 1 \end{pmatrix}$, so as to obtain
\[
\left\{\begin{pmatrix} a-ib & 0 \\ 0 & a+ib \end{pmatrix}\ \Big|\ a^2+b^2=1,\ a,b\in K\right\}\ =
\left\{\begin{pmatrix} t & 0 \\ 0 & t^{-1} \end{pmatrix}\ \Big|\ t\in K^\times\right\}.
\]
Since $\SO_2(K)$ is a normal subgroup of $\O_2(K)$ with index $2$, the claim follows by \cite[\S3]{Nagata}.

If $Z$ is $2\times n$ matrix of indeterminates over a field $K$ of odd characteristic, the fact that~$\O_2(K)$ is linearly reductive may be used to conclude that $K[Z^{\tr}Z]$ is a direct summand of the polynomial ring $K[Z]$ as a $K[Z^{\tr}Z]$-module. However, given the arithmetic nature of the paper, it is more advantageous to construct an explicit splitting over $\ZZ_2[i]$, as we carry out in the remark below, and more generally in the proof of Theorem~\ref{theorem:direct:summand}~\eqref{minors:c}.
\end{remark}

\begin{remark}
\label{remark:dirsum}
Let $R$ be the hypersurface $\ZZ[X]/(\det X)$, where $X$ is a symmetric $3\times 3$ matrix of indeterminates. Then $R$ may be identified with $\ZZ[Z^{\tr}Z]$, for $Z$ a $2\times 3$ matrix of indeterminates. Set $T\colonequals\ZZ[Z]$. Using arguments from the proof of Theorem~\ref{theorem:symmetric:3:odd}, we show that~$R_2=R[1/2]$ is a direct summand of $T_2=T[1/2]$ as an $R_2$-module.

Take $i=\sqrt{-1}$ in $\CC$. It suffices to prove that $R_2[i]$ is a direct summand of $T_2[i]$ as an~$R_2[i]$-module: indeed, if $\rho\colon T_2[i]\to R_2[i]$ is a splitting of $R_2[i]\into T_2[i]$, then
\[
\frac{1}{2}(\rho+\bar{\rho})\colon T_2\to R_2,
\]
with $\bar{\rho}$ denoting the complex conjugate, is a splitting of $R_2\into T_2$. In the ring $T_2[i]$, set
\begin{alignat*}{3}
 u\ &=\ z_{11}+iz_{21}, \qquad\qquad & x\ &=\ z_{11}-iz_{21},\\
 v\ &=\ z_{12}+iz_{22}, \qquad\qquad & y\ &=\ z_{12}-iz_{22},\\
 w\ &=\ z_{13}+iz_{23}, \qquad\qquad & z\ &=\ z_{13}-iz_{23},
\end{alignat*}
so that $T_2[i]=\ZZ_2[i][u,v,w,x,y,z]$. But $R_2[i]$ is the $\ZZ_2[i]$-algebra generated by the entries of 
\begin{multline*}
Z^{\tr}Z\ =\ 
\begin{pmatrix}
z_{11}^2+z_{21}^2 & z_{11}z_{12}+z_{21}z_{22} & z_{11}z_{13}+z_{21}z_{23}\\
z_{11}z_{12}+z_{21}z_{22} & z_{12}^2+z_{22}^2 & z_{12}z_{13}+z_{22}z_{23}\\
z_{11}z_{13}+z_{21}z_{23} & z_{12}z_{13}+z_{22}z_{23} & z_{13}^2+z_{23}^2
\end{pmatrix}\\
=\ \frac{1}{2}
\begin{pmatrix}
2ux & uy+vx & uz+wx\\ uy+vx & 2vy & vz+wy\\ uz+wx & vz+wy & 2wz
\end{pmatrix}.
\end{multline*}
The proof of Theorem~\ref{theorem:symmetric:3:odd} shows that $R_2[i]$ is a direct summand of $T_2[i]$ as an~$R_2[i]$-module.
\end{remark} 

To study the case of symmetric matrices in characteristic $2$, we record the following variant of Lemma~\ref{lemma:mainlemma}:

\begin{lemma}
\label{lemma:mainlemma:symmetric}
Let $S\colonequals\ZZ[\bsx]$ be a polynomial ring in the indeterminates $\bsx\colonequals x_0,\dots,x_d$. Fix a prime integer $p>0$, and let $f(\bsx)\in S$ be a polynomial of the form
\[
f(\bsx)\ =\ g(\bsx) + ph(\bsx),
\]
where $g(\bsx)\in(x_0,\dots,x_m)S$ for some fixed integer $m<d$, and $h(\bsx)\in S$ is a polynomial in the indeterminates $x_{m+1},\dots,x_d$. Set $R\colonequals S/(f(\bsx))$, and let $\phi_p$ denote the $p$-derivation of~$P_S$ as in Theorem~\ref{theorem:lifting:trace}. Then, if the local cohomology element
\[
\left[\frac{\phi_p\big(g(\bsx)\big)}{(x_0\cdots x_m)^p}\right]\ \in\ H^{m+1}_{(x_0,\,\dots,\,x_m)}(R/pR)
\]
is nonzero, so is the element
\[
\left[\frac{\phi_p\big(f(\bsy)-f(\bsx)\big)}{\prod_{i=0}^d (y_i-x_i)^p}\right]\ \in\ H^{d+1}_{\Delta_{R}}(P_R/pP_R).
\]
\end{lemma}

\begin{proof}
Suppose the displayed element of~$H^{d+1}_{\Delta_{R}}(P_R/pP_R)$ is zero. Then there exists an integer $k\ge 0$ such that
\begin{multline*}
\frac{\Lambda_p\big(f(\bsy)-f(\bsx)\big)-\big(f(\bsy)-f(\bsx)\big)^p}{p}(y_0-x_0)^k\cdots(y_d-x_d)^k\\
\in\ \big(p,\ f(\bsy),\ f(\bsx),\ (y_0-x_0)^{p+k},\ \dots,\ (y_d-x_d)^{p+k}\big)P_{S}.
\end{multline*}
Specialize $x_i\mapsto 0$ for $0\le i\le m$, in which case $f(\bsx)$ specializes to $ph(\bsx)$, so
\begin{multline*}
\frac{\Lambda_p\big(f(\bsy)-ph(\bsx)\big)-\big(f(\bsy)-ph(\bsx)\big)^p}{p}\ y_0^k\cdots y_m^k\ (y_{m+1}-x_{m+1})^k\cdots(y_d-x_d)^k\\
\in\ \big(p,\ f(\bsy),\ y_0^{p+k},\ \dots,\ y_m^{p+k},\ (y_{m+1}-x_{m+1})^{p+k},\ \dots,\ (y_d-x_d)^{p+k}\big).
\end{multline*}
Since $y_{m+1}-x_{m+1},\ \dots,\ y_d-x_d$ are algebraically independent over 
\[
\ZZ[\bsy]/(p,\ f(\bsy),\ y_0^{p+k},\ \dots,\ y_m^{p+k}),
\]
it follows that
\begin{multline*}
\frac{\Lambda_p\big(f(\bsy)-ph(\bsx)\big)-\big(f(\bsy)-ph(\bsx)\big)^p}{p}\ y_0^k\cdots y_m^k\\
\in\ \big(p,\ f(\bsy),\ y_0^{p+k},\ \dots,\ y_m^{p+k},\ (y_{m+1}-x_{m+1})^p,\ \dots,\ (y_d-x_d)^p\big).
\end{multline*}
Next, specialize $x_i\mapsto y_i$ for $m+1\le i\le d$. Then $ph(\bsx)$ specializes to $ph(\bsy)$, giving
\[
\frac{{\Lambda}_p\big(g(\bsy)\big)-\big(g(\bsy)\big)^p}{p}\ y_0^k\cdots y_m^k\ \in\ \big(p,\ f(\bsy),\ y_0^{p+k},\ \dots,\ y_m^{p+k}\big)\ZZ[\bsy],
\]
where $\Lambda_p$ is the standard lift of Frobenius on $\ZZ[\bsy]$ with respect to $\bsy$. Renaming $y_i\mapsto x_i$ for each $i$, it follows that $\displaystyle \left[\frac{\phi_p\big(g(\bsx)\big)}{(x_0\cdots x_m)^p}\right]\in\ H^{m+1}_{(x_0,\,\dots,\,x_m)}(R/pR)$ is zero.
\end{proof}

\begin{theorem}
\label{theorem:symmetric:char:2}
Let $X$ be a symmetric $3\times 3$ matrix of indeterminates over $\ZZ$, and set $R$ to be the hypersurface $\ZZ[X]/(\det X)$. Then the Frobenius trace on $R/2R$ does not lift to a differential operator on $R/4R$.
\end{theorem}

\begin{proof}
Set $S\colonequals\ZZ[X]$. Modulo the ideal $(2,x_{11},x_{22},x_{33})$, the image of $X$ is a $3\times 3$ alternating matrix, and hence has determinant zero; it follows that
\[
\det X\ \in\ (2,\,x_{11},\,x_{22},\,x_{33})S.
\]
Indeed, $\det X=g(\bsx) + 2h(\bsx)$, where
\[
g(\bsx)\colonequals x_{11}x_{22}x_{33} - x_{11}x_{23}^2 - x_{22}x_{13}^2 - x_{33}x_{12}^2
\quad\text{ and }\quad h(\bsx)\colonequals x_{12}x_{13}x_{23}.
\]
In light of Lemma~\ref{lemma:mainlemma:symmetric}, it suffices to check that
\[
\left[\frac{\phi_2\big(g(\bsx)\big)}{(x_{11}x_{22}x_{33})^2}\right]\ \in\ H^3_{(x_{11},\,x_{22},\,x_{33})}(R/2R)
\]
is nonzero. We shall go a step further and prove that while $\det X\equiv g(\bsx)\mod 2$,
\[
\left[\frac{\phi_2\big(\det X\big)}{(x_{11}x_{22}x_{33})^2}\right]\ =\ 0\ \neq\ \left[\frac{\phi_2\big(g(\bsx)\big)}{(x_{11}x_{22}x_{33})^2}\right].
\]
It is a straightforward calculation that modulo the ideal $(x_{11}^2,\ x_{22}^2,\ x_{33}^2)R/2R$, one has
\[
\phi_2\big(\det X\big)\ \equiv\ x_{11}x_{23}^2x_{22}x_{13}^2 + x_{12}^2x_{13}^2x_{23}^2
\quad\text{ and }\quad \phi_2\big(g(\bsx)\big)\ \equiv\ x_{11}x_{23}^2x_{22}x_{13}^2.
\]
In the ring $R/2R$, we set
\[
a\colonequals x_{11}x_{23}^2,\quad b\colonequals x_{22}x_{13}^2,\quad c\colonequals x_{33}x_{12}^2,\quad d\colonequals x_{11}x_{22}x_{33},\quad e\colonequals x_{12}x_{13}x_{23}. 
\]
Note that $a+b+c=d$ and $abc=de^2$ in $R/2R$. Working modulo $(x_{11}^3,\ x_{22}^3,\ x_{33}^3)R/2R$,
\[
\phi_2\big(\det X\big)\, x_{11}x_{22}x_{33}\ \equiv\ (ab+e^2)d\ \equiv\ ab(d+c)\ \equiv\ ab(a+b)\ \equiv\ a^3+b^3+(a+b)^3.
\]
Since $a\in (x_{11})R/2R$, and $b\in (x_{22})R/2R$, and $a+b=c+d\in (x_{33})R/2R$, it follows that~$\phi_2\big(\det X\big)x_{11}x_{22}x_{33}\in (x_{11}^3,\ x_{22}^3,\ x_{33}^3)R/2R$. This proves the first assertion.

Next, suppose $\left[\frac{\phi_2\big(g(\bsx)\big)}{(x_{11}x_{22}x_{33})^2}\right]=0$. Then there exists an integer $k\ge0$ such that
\begin{equation}
\label{equation:symmetric:3x3:g}
x_{11}x_{23}^2x_{22}x_{13}^2\, (x_{11}x_{22}x_{33})^k\ \in\ \big(x_{11}^{2+k},\ x_{22}^{2+k},\ x_{33}^{2+k}\big)R/2R.
\end{equation}
Consider the $\ZZ^2$-grading on $R/2R$ defined by
\begin{alignat*}3
\deg x_{11}\ &=\ (2, 0), \qquad\qquad & \deg x_{23}\ &=\ (-1, 0),\\
\deg x_{22}\ &=\ (0, 2), \qquad\qquad & \deg x_{13}\ &=\ (0, -1),\\
\deg x_{33}\ &=\ (-2, -2), \qquad\qquad & \deg x_{12}\ &=\ (1, 1).
\end{alignat*}
The element on the left in~\eqref{equation:symmetric:3x3:g} has degree $(0,0)$, so we work in the subring ${[R/2R]}_{(0,0)}$, which is the $\ZZ/2\ZZ$-algebra generated by
\[
x_{11}x_{23}^2,\quad x_{22}x_{13}^2,\quad x_{33}x_{12}^2,\quad x_{11}x_{22}x_{33},\quad x_{12}x_{13}x_{23}.
\]
Using $c=d-a-b$, this subring may be identified with
\[
B\colonequals \ZZ/2\ZZ[a,b,d,e]/\big(ab(d-a-b)-de^2\big).
\]
In the hypersurface $B$,~\eqref{equation:symmetric:3x3:g} implies that
\[
abd^k\ \in\ (a,d)^{2+k} + (b,d)^{2+k} + (c,d)^{2+k}\ \subseteq\ (a^2,\ b^2,\ d^{k+1}).
\]
But the image of $abd^k$ is nonzero in
\[
B/(a^2, b^2, d^{k+1}, ab-e^2)\ =\ \ZZ/2\ZZ[a,b,d,e]/(a^2,\ b^2,\ d^{k+1},\ ab-e^2),
\]
which is a contradiction.
\end{proof}

Lastly, we examine the existence of Frobenius lifts for hypersurfaces defined by determinants of symmetric matrices of indeterminates:

\begin{theorem}
\label{theorem:symmetric:frobenius}
Let $X$ be an $n\times n$ symmetric matrix of indeterminates over $\ZZ$. Set~$S\colonequals\ZZ[X]$ and~$R\colonequals S/(\det X)$.
\begin{enumerate}[\,\rm(a)]
\item\label{nxn:sym:a} If $n=3$ and $p$ is an odd prime integer, then the Frobenius endomorphism on $R/pR$ lifts to an endomorphism of $R/p^2R$.
\item\label{nxn:sym:b} If $n\ge 3$, then the Frobenius endomorphism on $R/2R$ does not lift to an endomorphism of $R/4R$.
\item\label{nxn:sym:c} For $n\ge4$, and $p$ an odd prime integer, the Frobenius endomorphism on $R/pR$ does not lift to an endomorphism of $R/p^2R$.
\end{enumerate}
\end{theorem}

\begin{proof}
\eqref{nxn:sym:a} In the case $n=3$, the ring $R[1/2]$ is a direct summand, as an $R[1/2]$-module, of a polynomial ring over $\ZZ[1/2]$, see Remark~\ref{remark:dirsum}. But then, for $p$ odd, the ring $R/p^2R$ is a direct summand, as an $R/p^2R$-module, of a polynomial ring over~$\ZZ/p^2\ZZ$. The existence of a Frobenius lift now follows using \cite[Lemma~4.1]{Zdanowicz}.

In the remaining cases, in view of Proposition~\ref{proposition:zdanowicz}, we need to verify that
\begin{equation}
\label{equation:frobenius:symmetric1}
\phi_p(\det X)\ \notin\ \bigg(p,\ \det X,\ \Big(\frac{\partial\det X}{\partial x_{ij}}\Big)^p : 1\le i\le j\le n\bigg)S.
\end{equation}
A partial derivative of the form $\partial\det X/\partial x_{ii}$ is the determinant of a size $n-1$ principal submatrix of $X$, whereas, for $i<j$, the partial derivative $\partial\det X/\partial x_{ij}$ is, aside from a sign, twice the determinant of a size $n-1$ submatrix. This necessitates the distinction between the cases where $p$ equals $2$, and where $p$ is an odd prime. 

For~\eqref{nxn:sym:b}, suppose~\eqref{equation:frobenius:symmetric1} fails for some $n\ge 3$. In view of the above paragraph, one has
\begin{equation}
\label{equation:frobenius:symmetric2}
\phi_2(\det X)\ \in\ \bigg(2,\ \det X,\ \Big(\frac{\partial\det X}{\partial x_{ii}}\Big)^p : 1\le i\le n\bigg)S.
\end{equation}
Specialize the symmetric matrix $X$ as
\[
X\mapsto \begin{pmatrix}X' & 0\\ 0 & I\end{pmatrix},
\qquad\text{ where }\ X'\colonequals
\begin{pmatrix}
0 & x_{12} & x_{13}\\
x_{12} & 0 & x_{23}\\
x_{13} & x_{23} & 0
\end{pmatrix},
\]
and $I$ is the size $n-3$ identity matrix. Then $\det X$ specializes to $\det X'=2x_{12}x_{13}x_{23}$, so~$\phi_2(\det X)$ specializes to
\[
\phi_2(\det X')\ =\ \frac{1}{2} \left(2x_{12}^2x_{13}^2x_{23}^2 - (2x_{12}x_{13}x_{23})^2\right)\ =\ -x_{12}^2x_{13}^2x_{23}^2.
\]
With $S'$ denoting the image of $S$, the ideal membership~\eqref{equation:frobenius:symmetric2} implies that
\[
x_{12}^2x_{13}^2x_{23}^2\ \in\ \left(x_{12}^2,\ x_{13}^2,\ x_{23}^2\right)^{[2]}S'/2S',
\]
which is a contradiction.

For~\eqref{nxn:sym:c}, suppose~\eqref{equation:frobenius:symmetric1} fails for some $n\ge 4$, and $p$ odd. Then
\begin{equation}
\label{equation:frobenius:symmetric3}
\phi_p(\det X)\ \in\ I_{n-1}(X)^{[p]}\ S/(p,\,\det X)S,
\end{equation}
where $I_{n-1}(X)$ denotes the ideal generated by the size $n-1$ minors of $X$. Specialize
\[
X\mapsto \begin{pmatrix}X' & 0\\ 0 & I\end{pmatrix},
\qquad\text{ where }\ X'\colonequals
\begin{pmatrix}
0 & x_{12} & x_{13} & x_{14}\\
x_{12} & 0 & x_{23} & x_{24}\\
x_{13} & x_{23} & 0 & x_{34}\\
x_{14} & x_{24} & x_{34} & 0
\end{pmatrix},
\]
and $I$ is the size $n-4$ identity matrix. Then $I_{n-1}(X)S'=I_3(X')$, with~$S'$ denoting the image of $S$, and~\eqref{equation:frobenius:symmetric3} specializes to
\begin{equation}
\label{equation:frobenius:symmetric4}
\phi_p(\det X')\ \in\ I_3(X')^{[p]}\ S'/(p,\,\det X')S'.
\end{equation}
Consider the $\ZZ^4$-grading on $S'/(p,\,\det X')S'$ defined by
\begin{alignat*}3
\deg x_{12}\ &=\ e_1+e_4, \qquad\qquad & \deg x_{34}\ &=\ e_4-e_1,\\
\deg x_{13}\ &=\ e_2+e_4, \qquad\qquad & \deg x_{24}\ &=\ e_4-e_2,\\
\deg x_{23}\ &=\ e_3+e_4, \qquad\qquad & \deg x_{14}\ &=\ e_4-e_3,
\end{alignat*}
under which $\det X'$ has degree $(0,0,0,4)$, and $\phi_p(\det X')$ has degree $(0,0,0,4p)$. Let $\Delta_{ij}$ denote the determinant of the submatrix of $X'$ obtained by deleting the $i$-th row and $j$-th column. Then
\[
\deg\Delta_{ij}\ =\ 4e_4-\deg x_{ij}\qquad\text{for }i<j,
\]
whereas
\begin{alignat*}3
\deg\Delta_{11}\ &=\ (-1,-1,1,3),\qquad\qquad & \deg\Delta_{33}\ &=\ (1,-1,-1,3)\\
\deg\Delta_{22}\ &=\ (-1,1,-1,3),\qquad\qquad & \deg\Delta_{44}\ &=\ (1,1,1,3).
\end{alignat*}
Hence a homogeneous equation for~\eqref{equation:frobenius:symmetric4} forces
\[
\phi_p(\det X')\ \in\ \big(x_{ij}^p\Delta_{ij}^p : 1\le i< j\le 4\big) S'/(p,\,\det X')S'.
\]
It is readily checked that
\[
x_{12}\Delta_{12}=x_{34}\Delta_{34},\qquad x_{13}\Delta_{13}=x_{24}\Delta_{24},\qquad x_{14}\Delta_{14}=x_{23}\Delta_{23},
\]
so 
\begin{equation}
\label{equation:frobenius:symmetric5}
\phi_p(\det X')\ \in\ \big(x_{12}^p\Delta_{12}^p,\ x_{13}^p\Delta_{13}^p) S'/(p,\,\det X')S'.
\end{equation}
Since the elements in question have degree of the form $(0,0,0,*)$, this ideal membership holds in the subring ${[S'/(p,\,\det X')S']}_{(0,0,0,*)}$, which may be identified with
\[ 
B\colonequals\ZZ/p\ZZ[a,b,c]/(a^2+b^2+c^2-2ab-2ac-2bc),
\]
where $a\colonequals x_{12}x_{34}$, $b\colonequals x_{13}x_{24}$, and $c\colonequals x_{14}x_{23}$.
In this ring,~\eqref{equation:frobenius:symmetric5} implies that
\[
\frac{1}{p}\big(a^{2p}+b^{2p}+c^{2p}-2a^pb^p-2a^pc^p-2b^pc^p\big)\ \in\ \big(a^2-ab-ac,\ b^2-ab-bc\big)^{[p]}B.
\]
Enlarge the ring $B$ by adjoining $u\colonequals\sqrt{a}$ and $v\colonequals\sqrt{b}$, in which case the defining equation of $B$ factors as
\[
\big(c-(u+v)^2\big)\big(c-(u-v)^2\big)
\]
We work modulo the first factor $c-(u+v)^2$, i.e., in the polynomial ring
\[
\ZZ/p\ZZ[u,v],
\]
where $c$ is identified with $(u+v)^2$. The ideal membership then implies that
\[
\frac{1}{p}\Big(u^{4p}+v^{4p}+(u+v)^{4p}-2u^{2p}v^{2p}-2u^{2p}(u+v)^{2p}-2v^{2p}(u+v)^{2p}\Big)
\]
is a linear combination of $(u^2v^2+u^3v)^p$ and $(u^2v^2+uv^3)^p$ with coefficients from $\ZZ/p\ZZ$. This is not possible, for example by examining the coefficient of $u^{3p-1}v^{p+1}$; the interested reader---if one remains---may verify that this coefficient is $8$. 
\end{proof}

\section{Pfaffian, determinantal, and symmetric determinantal rings}
\label{section:minors}

We extend some earlier results to rings defined by minors and Pfaffians of arbitrary size; the following theorem subsumes Theorem~\ref{theorem:frob} and Corollary~\ref{corollary:not:summand}.

\begin{theorem}
\label{theorem:direct:summand}
Let $V$ denote either the ring of integers, or the ring of $p$-adic integers $\widehat{\ZZ_{(p)}}$.

\begin{enumerate}[\,\rm(a)]
\item\label{minors:a} Let $R\colonequals V[X]/\Pf_t(X)$, where $X$ is an $n\times n$ alternating matrix of indeterminates, and~$\Pf_t(X)$ the ideal generated by the Pfaffians of the size~$t$ principal submatrices of $X$, where $t$ is even, and~$n\ge t\ge 2$. Then $R$ is a direct summand of a polynomial ring over~$V$, as an $R$-module, if and only $t=2$.

\item\label{minors:b} Let $R\colonequals V[X]/I_t(X)$, where $X$ is an $m\times n$ matrix of indeterminates, and $I_t(X)$ the ideal generated by the size $t$ minors of $X$, where $\min\{m,n\}\ge t\ge 2$. Then $R$ is a direct summand of a polynomial ring over $V$, as an $R$-module, if and only $t=2$.

\item\label{minors:c} Let $R\colonequals V[X]/I_t(X)$, where $X$ is a symmetric $n\times n$ matrix of indeterminates, and~$I_t(X)$ the ideal generated by the size $t$ minors of $X$, where $n\ge t\ge 2$. Then $R$ is a direct summand of a polynomial ring over $V$, as an $R$-module, if and only $t=2$ or if $t=3$ and the prime $p$ is odd in the case $V=\widehat{\ZZ_{(p)}}$.
\end{enumerate}

In each of the case above where $R$ is not a direct summand of a polynomial ring, the Frobenius endomorphism on~$R/pR$ does not lift to an endomorphism of $R/p^2R$.
\end{theorem}

\begin{proof}
In each case, when proving that $R$ is \emph{not} a direct summand of any polynomial ring over $V$, or that the Frobenius endomorphism on~$R/pR$ does not lift to an endomorphism of~$R/p^2R$, it suffices to consider the case where $V=\widehat{\ZZ_{(p)}}$. Next, reduce to the case where $X$ is a $t\times t$ matrix, since such a matrix may be enlarged by adding additional rows or columns of indeterminates, with a retraction from the larger determinantal ring to the hypersurface obtained by killing the additional indeterminates. Thus, we may assume in each case that~$R$ is a hypersurface over $V$. The three cases then follow from Theorem~\ref{theorem:pfaffian:frobenius}, Theorem~\ref{theorem:determinant:frobenius}, and Theorem~\ref{theorem:symmetric:frobenius} respectively.

It remains to settle the cases where $R$ is a direct summand of a polynomial ring over $V$. When~$t=2$ in~\eqref{minors:a}, one simply has $R=V$. In case~\eqref{minors:b}, when $t=2$, consider the polynomial ring $T\colonequals V[y_1,\dots,y_m,z_1,\dots,z_n]$. It is readily seen that $x_{ij}\mapsto y_iz_j$ induces an inclusion of~$V$-algebras. The ring $T$ admits a $\ZZ$-grading where $V$ has degree $0$, each $y_i$ has degree $1$, and each $z_j$ has degree $-1$. The projection to the degree $0$ component
\[
T_0\ =\ V[y_iz_j : 1\le i\le m,\ 1\le j\le n]\ \cong\ R
\] gives the $R$-linear splitting.

In case~\eqref{minors:c}, when $t=2$, the ring $R$ may be identified with the Veronese subring
\[
V[y_iy_j:1\le i\le j\le n]
\]
which is a direct summand of the polynomial ring $V[y_1,\dots,y_n]$. Lastly, when $t=3$, the ring~$\ZZ_2[X]/I_3(X)$ may be identified with $\ZZ_2[Z^{\tr}Z]$, where $Z$ is a $2\times n$ matrix of indeterminates; we show that
\[
\ZZ_2[Z^{\tr}Z]\ \to\ \ZZ_2[Z]
\]
is $\ZZ_2[Z^{\tr}Z]$-split. As in Remark~\ref{remark:dirsum}, one may adjoin $i=\sqrt{-1}$ and perform a change of variables, which reduces the problem to proving that the $\ZZ_2$-algebra generated by the entries of the product matrix
\[
\begin{pmatrix}
z_1 & y_1\\
z_2 & y_2\\
\vdots & \vdots\\
z_n & y_n\\
\end{pmatrix}
\begin{pmatrix}
y_1 & y_2 & \hdots & y_n\\
z_1 & z_2 & \hdots & z_n
\end{pmatrix}
\ =\
\begin{pmatrix}
2y_1z_1 & y_1z_2+y_2z_1 & \hdots & y_1z_n+y_nz_1\\
y_1z_2+y_2z_1 & 2y_2z_2 & \hdots & y_2z_n+y_nz_2\\
\vdots & \vdots & \ddots & \vdots\\
y_1z_n+y_nz_1 & y_2z_n+y_nz_2 & \hdots & 2y_nz_n\\
\end{pmatrix}
\]
is a direct summand of the polynomial ring $T_2\colonequals\ZZ_2[y_1,\dots,y_n,z_1,\dots,z_n]$. For this, as in the proof of Theorem~\ref{theorem:symmetric:3:odd}, note that the ring
\[
B\colonequals \ZZ_2[y_iz_j : i\le i,j\le n]
\]
is a direct summand $T_2$ as a $B$-module, and that the involution
\[
y_i\mapsto z_i
\]
on $T_2$ restricts to an action on the subring $B$, with the invariant ring being the $\ZZ_2$-algebra generated by the entries of the product matrix.
\end{proof}

\section*{Acknowledgments}

The first author thanks Alessandro De Stefani and Elo\'isa Grifo for many helpful discussions on $p$-derivations.



\begin{thebibliography}{BBLSZ}

\bibitem[BN]{BZN}
D.~Ben-Zvi and T.~Nevins, \emph{Cusps and $\calD$-modules}, J. Amer. Math. Soc.~\textbf{17} (2004), 155--179. 

\bibitem[BGG]{BGG}
I.~N.~Bernstein, I.~M.~Gel'fand, and S.~I.~Gel'fand, \emph{Differential operators on a cubic cone}, Russian Math. Surveys~\textbf{27} (1972), 169--174.

\bibitem[Bh]{Bhatt:doc}
B.~Bhatt, \emph{Torsion in the crystalline cohomology of singular varieties}, Doc. Math.~\textbf{19} (2014), 673--687.

\bibitem[BBLSZ]{BBLSZ}
B.~Bhatt, M.~Blickle, G.~Lyubeznik, A.~K.~Singh, and W.~Zhang, \emph{Local cohomology modules of a smooth $\mathbb Z$-algebra have finitely many associated primes}, Invent. Math.~\textbf{197} (2014), 509--519.

\bibitem[BJNB]{BJNB}
H.~Brenner, J.~Jeffries, and L.~N\'{u}\~{n}ez-Betancourt, \emph{Quantifying singularities with differential operators}, Adv. Math.~\textbf{358} (2019), 106843, 89~pp.

\bibitem[BTLM]{BTLM}
A.~Buch, J.~F.~Thomsen, N.~Lauritzen, and V.~Mehta, \emph{The Frobenius morphism on a toric variety}, Tohoku Math. J.~\textbf{49} (1997), 355--366.

\bibitem[Bu]{Buium}
A.~Buium, \emph{Differential characters of abelian varieties over p-adic fields}, Invent. Math.~\textbf{122} (1995), 309--340.

\bibitem[DCP]{DCP}
C.~De~Concini and C.~Procesi, \emph{A characteristic free approach to invariant theory}, Adv. Math.~\textbf{21} (1976), 330--354. 


\bibitem[Gr]{EGA4}
A.~Grothendieck, \emph{\'El\'ements de g\'eom\'etrie alg\'ebrique IV, \'Etude locale des sch\'emas et des morphismes de sch\'emas IV}, Inst. Hautes \'Etudes Sci. Publ. Math.~\textbf{32} (1967), 5--361.


\bibitem[HE]{Hochster-Eagon}
M.~Hochster and J.~A.~Eagon, \emph{Cohen-Macaulay rings, invariant theory, and the generic perfection of determinantal loci}, Amer. J. Math.~\textbf{93} (1971), 1020--1058.

\bibitem[HR]{Hochster-Roberts}
M.~Hochster and J.~L.~Roberts, \emph{Rings of invariants of reductive groups acting on regular rings are Cohen-Macaulay}, Adv. Math.~\textbf{13} (1974), 115--175.

\bibitem[Hu]{Huneke:Sundance}
C.~Huneke, \emph{Problems on local cohomology}, in: Free resolutions in commutative algebra and algebraic geometry (Sundance, Utah, 1990), 93--108, Res. Notes Math.~\textbf{2}, Jones and Bartlett, Boston, MA, 1992.

\bibitem[Je]{Jeffries}
J.~Jeffries, \emph{Derived functors of differential operators}, Int. Math. Res. Not. IMRN~\textbf{7} (2021), 4920--4940.

\bibitem[Jo]{Joyal}
A.~Joyal, \emph{$\delta$-anneaux et vecteurs de Witt}, C. R. Math. Rep. Acad. Sci. Canada~\textbf{7} (1985), 177--182.

\bibitem[Ka1]{Kantor1}
J.-M.~Kantor, \emph{Formes et op\'erateurs diff\'erentiels sur les espaces analytiques complexes}, Bull. Soc. Math. France M\'em.
~\textbf{53} (1977), 5--80.

\bibitem[Ka2]{Kantor2}
J.-M.~Kantor, \emph{D\'erivations sur les singularit\'es quasi-homog\`enes: cas des courbes}, C. R. Acad. Sci. Paris S\'er. A-B~\textbf{287} (1978) 1117--1119; \emph{Rectificatif}~\textbf{288} (1979), 697.

\bibitem[LS]{LS}
T.~Levasseur and J.~T.~Stafford, \emph{Rings of differential operators on classical rings of invariants}, Mem. Amer. Math. Soc.~\textbf{81} (1989), no. 412.

\bibitem[Ly]{Lyubeznik:Crelle}
G.~Lyubeznik, \emph{$F$-modules: applications to local cohomology and $D$-modules in characteristic $p>0$}, J. Reine Angew. Math.~\textbf{491} (1997), 65--130.

\bibitem[LSW]{LSW}
G.~Lyubeznik, A.~K.~Singh, and U.~Walther, \emph{Local cohomology modules supported at determinantal ideals}, J. Eur. Math. Soc. (JEMS)~\textbf{18} (2016), 2545--2578.

\bibitem[Ma]{Mal}
D.~Mallory, \emph{Bigness of the tangent bundle of del Pezzo surfaces and $D$-simplicity}, Algebra Number Theory~\textbf{15} (2021), 2019--2036. 

\bibitem[MV]{MV}
I.~M.~Musson and M.~Van~den~Bergh, \emph{Invariants under tori of rings of differential operators and related topics}, Mem. Amer. Math. Soc.~\textbf{136} (1998), no. 650.

\bibitem[Na]{Nagata}
M.~Nagata, \emph{Complete reducibility of rational representations of a matric group}, J. Math. Kyoto Univ.~\textbf{1} (1961), 87--99.

\bibitem[Sc1]{Schwarz:ICM}
G.~W.~Schwarz, \emph{Invariant differential operators}, Proceedings of the International Congress of Mathematicians, (Z\"urich, 1994),
pp.~333--341, Birkh\"auser, Basel, 1995.

\bibitem[Sc2]{Schwarz:ASENS}
G.~W.~Schwarz, \emph{Lifting differential operators from orbit spaces}, Ann. Sci. \'Ecole Norm. Sup.~(4)~\textbf{28} (1995), 253--305.

\bibitem[Si]{Singh:MRL}
A.~K.~Singh, \emph{$p$-torsion elements in local cohomology modules}, Math. Res. Lett.~\textbf{7} (2000), 165--176.

\bibitem[SW]{SinghWalther:Bock}
A.~K.~Singh and U.~Walther, \emph{Bockstein homomorphisms in local cohomology}, J. Reine Angew. Math.~\textbf{655} (2011), 147--164. 

\bibitem[Sm]{Smith}
K.~E.~Smith, \emph{The $D$-module structure of $F$-split rings}, Math. Res. Lett.~\textbf{2} (1995), 377--386.

\bibitem[SV]{SV}
K.~E.~Smith and M.~Van~den~Bergh, \emph{Simplicity of rings of differential operators in prime characteristic}, Proc. London Math. Soc.~\textbf{75} (1997), 32--62. 

\bibitem[Zd]{Zdanowicz}
M.~Zdanowicz, \emph{Liftability of singularities and their Frobenius morphism modulo $p^2$}, Int. Math. Res. Not. IMRN~\textbf{14} (2018), 4513--4577. 

\end{thebibliography}
\end{document}